\documentclass{amsart}
\usepackage{graphicx}
\usepackage[all]{xy}
\usepackage{amsthm}
\usepackage{amssymb}
\usepackage{cite}

\theoremstyle{definition}
\newtheorem{1def}{Definition}[section]
\newtheorem{exa}[1def]{Example}

\theoremstyle{plain}
\newtheorem{thm}[1def]{Theorem}
\newtheorem{lem}[1def]{Lemma}
\newtheorem{pro}[1def]{Proposition}
\newtheorem{cor}[1def]{Corollary}

\theoremstyle{remark}
\newtheorem{rmk}[1def]{Remark}

\numberwithin{equation}{section}

\begin{document}
\title{Upper triangular matrices and Billiard Arrays}
\author{Yang Yang}
\address{Department of Mathematics, University of Wisconsin, Madison, WI 53706, USA}
\email{yyang@math.wisc.edu}

\begin{abstract}
Fix a nonnegative integer $d$, a field $\mathbb{F}$, and a vector space $V$ over $\mathbb{F}$ with dimension $d+1$. Let $T$ denote an invertible upper triangular matrix in ${\rm Mat}_{d+1}(\mathbb{F})$. Using $T$ we construct three flags on $V$. We find a necessary and sufficient condition on $T$ for these three flags to be totally opposite. In this case, we use these three totally opposite flags to construct a Billiard Array $B$ on $V$. It is known that $B$ is determined up to isomorphism by a certain triangular array of scalar parameters called the $B$-values. We compute these $B$-values in terms of the entries of $T$. We describe the set of isomorphism classes of Billiard Arrays in terms of upper triangular matrices.


\end{abstract}

\maketitle

\section{Introduction}
This paper is about a connection between upper triangular matrices and Billiard Arrays. The Billiard Array concept was introduced in \cite{paul}. This concept is closely related to the equitable presentation of $U_q(\mathfrak{sl}_2)$\cite{equit, paul}. For more information about the equitable presentation, see \cite{alnajjar, neubauer, huang, irt, uqsl2hat, nonnil, qtet, tersym, uawe,fduq,boyd}.

We now summarize our results. Fix a nonnegative integer $d$, a field $\mathbb{F}$, and a vector space $V$ over $\mathbb{F}$ with dimension $d+1$. Let $T$ denote an invertible upper triangular matrix in ${\rm Mat}_{d+1}(\mathbb{F})$. View $T$ as the transition matrix from a basis $\{u_i\}_{i=0}^{d}$ of $V$ to a basis $\{v_i\}_{i=0}^{d}$ of $V$. Using $T$ we construct three flags $\{U_{i}\}_{i=0}^{d}$, $\{U_{i}^{\prime}\}_{i=0}^{d}$, $\{U_{i}^{\prime\prime}\}_{i=0}^{d}$ on $V$ as follows. For $0\le i\le d$,
\begin{equation*}
U_{i}=\mathbb{F}u_0+\mathbb{F}u_1+\cdots +\mathbb{F}u_i=\mathbb{F}v_0+\mathbb{F}v_1+\cdots +\mathbb{F}v_i;
\end{equation*}
\begin{equation*}
U_{i}^{\prime}=\mathbb{F}u_d+\mathbb{F}u_{d-1}+\cdots +\mathbb{F}u_{d-i};
\end{equation*}
\begin{equation*}
U_{i}^{\prime\prime}=\mathbb{F}v_d+\mathbb{F}v_{d-1}+\cdots +\mathbb{F}v_{d-i}.
\end{equation*}

In our first main result, we find a necessary and sufficient condition (called very good) on $T$ for $\{U_{i}\}_{i=0}^{d}$, $\{U_{i}^{\prime}\}_{i=0}^{d}$, $\{U_{i}^{\prime\prime}\}_{i=0}^{d}$ to be totally opposite in the sense of \cite[Definition~12.1]{paul}.

In \cite[Theorem~12.7]{paul} it is shown how three totally opposite flags on $V$ correspond to a Billiard Array on $V$. Assume that the three flags $\{U_{i}\}_{i=0}^{d}$, $\{U_{i}^{\prime}\}_{i=0}^{d}$, $\{U_{i}^{\prime\prime}\}_{i=0}^{d}$ are totally opposite, and let $B$ denote the corresponding Billiard Array on $V$. By \cite[Lemma~19.1]{paul} $B$ is determined up to isomorphism by a certain triangular array of scalar parameters called the $B$-values. In our second main result, we compute these $B$-values in terms of the entries of $T$.

Let $\mathcal{T}_d(\mathbb{F})$ denote the set of very good upper triangular matrices in ${\rm Mat}_{d+1}(\mathbb{F})$. Define an equivalence relation $\sim$ on $\mathcal{T}_d(\mathbb{F})$ as follows. For $T, T^\prime\in \mathcal{T}_d(\mathbb{F})$, we declare $T\sim T^\prime$ whenever there exist invertible diagonal matrices $H, K \in {\rm Mat}_{d+1}(\mathbb{F})$ such that $T^{\prime}=HTK$. In our third main result, we display a bijection between the following two sets:
\begin{enumerate}
\item[\rm(i)] the equivalence classes for $\sim$ on $\mathcal{T}_d(\mathbb{F})$;
\item[\rm(ii)] the isomorphism classes of Billiard Arrays on $V$.
\end{enumerate}
We give a commutative diagram that illustrates our second and third main result. At the end of this paper, we give an example. In this example, we display a very good upper triangular matrix with entries given by $q$-binomial coefficients. We show that for the corresponding Billiard Array $B$, all the $B$-values are equal to $q^{-1}$.

The paper is organized as follows. Section 2 contains some preliminaries. Section 3 contains necessary facts about decompositions and flags. Section 4 is devoted to the correspondence between very good upper triangular matrices and totally opposite flags. This section contains our first main result. Section 5 contains necessary facts about Billiard Arrays. In Sections 6--8 we obtain our second and third main results. In Section 9, we display an example to illustrate our theory.

\section{Preliminaries}
Throughout the paper, we fix the following notation. Let $\mathbb{R}$ denote the field of real numbers. Recall the ring of integers $\mathbb{Z}=\{0,\pm 1,\pm 2,\dots\}$ and the set of natural numbers $\mathbb{N}=\{0,1,2,\dots\}$. Fix $d\in \mathbb{N}$. Let $\{x_i\}_{i=0}^d$ denote a sequence. We call $x_i$ the {\it $i$-component} of the sequence. By the {\it inversion} of the sequence $\{x_i\}_{i=0}^d$ we mean the sequence $\{x_{d-i}\}_{i=0}^d$. Let $\mathbb{F}$ denote a field. Let $V$ denote a vector space over $\mathbb{F}$ with dimension $d+1$. Let ${\rm Mat}_{d+1}(\mathbb{F})$ denote the $\mathbb{F}$-algebra consisting of the $d+1$ by $d+1$ matrices that have all entries in $\mathbb{F}$. We index the rows and columns by $0,1,\dots ,d$. Let $I$ denote the identity matrix in ${\rm Mat}_{d+1}(\mathbb{F})$.

\section{Decompositions and Flags}
In this section, we review some basic facts about decompositions and flags.
\begin{1def}
By a {\it decomposition} of $V$ we mean a sequence $\{V_i\}_{i=0}^d$ consisting of one-dimensional subspaces of $V$ such that $V=\sum_{i=0}^{d}V_i$ (direct sum).
\end{1def}
\begin{rmk}
For a decomposition of $V$, its inversion is a decomposition of $V$.
\end{rmk}
\begin{exa}
Choose a basis $\{f_i\}_{i=0}^{d}$ of $V$. For $0\le i\le d$, define $V_i=\mathbb{F}f_i$. Then $\{V_i\}_{i=0}^d$ is a decomposition of $V$.
\end{exa}
\begin{1def}
Referring to Example 3.3, we say that the decomposition $\{V_i\}_{i=0}^d$ is {\it induced} by the basis $\{f_i\}_{i=0}^{d}$.
\end{1def}
\begin{1def}
By a {\it flag} on $V$, we mean a sequence $\{W_{i}\}_{i=0}^{d}$ of subspaces of $V$ such that $W_{i}$ has dimension $i+1$ for $0\le i \le d$ and $W_{i-1}\subseteq W_{i}$ for $1\le i \le d$.
\end{1def}
\begin{exa}
Let $\{V_i\}_{i=0}^d$ denote a decomposition of $V$. For $0\le i\le d$, define $W_i=V_0+V_1+\dots+V_i$. Then $\{W_{i}\}_{i=0}^{d}$ is a flag on $V$.
\end{exa}
\begin{1def}
Referring to Example 3.6, we say that the flag $\{W_{i}\}_{i=0}^{d}$ is {\it induced} by the decomposition $\{V_i\}_{i=0}^d$.
\end{1def}

\begin{1def}
Consider a basis of $V$. That basis induces a decomposition of $V$, which in turn induces a flag on $V$. We say that flag is {\it induced} by the given basis.
\end{1def}

\begin{lem}
\cite[Section~6]{paul} Suppose that we are given two flags on $V$, denoted by $\{W_{i}\}_{i=0}^{d}$ and $\{W_{i}^{\prime}\}_{i=0}^{d}$. Then the following are equivalent:
\begin{enumerate}
\item[(\rm i)] $W_i\cap W_{j}^{\prime}=0$ for $i+j <d \ (0\le i,j \le d)$;
\item[(\rm ii)] there exists a decomposition $\{V_i\}_{i=0}^d$ of $V$ that induces $\{W_{i}\}_{i=0}^{d}$ and whose inversion induces $\{W_{i}^{\prime}\}_{i=0}^{d}$.
\end{enumerate}
Moreover, suppose {\rm (i), (ii)} hold. Then $V_i=W_{i}\cap W^\prime_{d-i}$ for $0\le i\le d$.
\end{lem}

\begin{1def}
Referring to Lemma 3.9, the flags $\{W_{i}\}_{i=0}^{d}$ and $\{W_{i}^{\prime}\}_{i=0}^{d}$ are called {\it opposite} whenever (i), (ii) hold.
\end{1def}
We mention a variation on Lemma 3.9.
\begin{lem}
\cite[Section~6]{paul} Suppose that we are given two flags on $V$, denoted by $\{W_{i}\}_{i=0}^{d}$ and $\{W_{i}^{\prime}\}_{i=0}^{d}$ . Then they are opposite if and only if $W_i\cap W_{j}^{\prime}=0$ for $i+j =d-1 \ (0\le i,j \le d-1)$.
\end{lem}

\begin{1def}
Suppose that we are given three flags on $V$, denoted by $\{W_{i}\}_{i=0}^{d}$, $\{W_{i}^{\prime}\}_{i=0}^{d}$, $\{W_{i}^{\prime\prime}\}_{i=0}^{d}$. These flags are said to be {\it totally opposite} whenever $W_{d-r}\cap W_{d-s}^{\prime}\cap W_{d-t}^{\prime\prime}=0$ for all $r,s,t$ $(0\le r,s,t\le d)$ such that $r+s+t>d$.
\end{1def}

\begin{lem}
\cite[Theorem~12.3]{paul} Suppose that we are given three flags on $V$, denoted by $\{W_{i}\}_{i=0}^{d}$, $\{W_{i}^{\prime}\}_{i=0}^{d}$, $\{W_{i}^{\prime\prime}\}_{i=0}^{d}$. Then the following are equivalent:
\begin{enumerate}
\item[\rm(i)] The flags $\{W_{i}\}_{i=0}^{d}$, $\{W_{i}^{\prime}\}_{i=0}^{d}$, $\{W_{i}^{\prime\prime}\}_{i=0}^{d}$ are totally opposite.
\item[\rm(ii)] For $0\le n \le d$, the sequences $\{W_{i}\}_{i=0}^{d-n}$, $\{W_{d-n}\cap W_{n+i}^{\prime}\}_{i=0}^{d-n}$, $\{W_{d-n}\cap W_{n+i}^{\prime\prime}\}_{i=0}^{d-n}$ are mutually opposite flags on $W_{d-n}$.
\item[\rm(iii)]For $0\le n \le d$, the sequences $\{W_{i}^{\prime}\}_{i=0}^{d-n}$, $\{W_{d-n}^{\prime}\cap W_{n+i}^{\prime\prime}\}_{i=0}^{d-n}$, $\{W_{d-n}^{\prime}\cap W_{n+i}\}_{i=0}^{d-n}$ are mutually opposite flags on $W_{d-n}^{\prime}$.
\item[\rm(iv)]For $0\le n \le d$, the sequences $\{W_{i}^{\prime\prime}\}_{i=0}^{d-n}$, $\{W_{d-n}^{\prime\prime}\cap W_{n+i}\}_{i=0}^{d-n}$, $\{W_{d-n}^{\prime\prime}\cap W_{n+i}^{\prime}\}_{i=0}^{d-n}$ are mutually opposite flags on $W_{d-n}^{\prime\prime}$.
\end{enumerate}
\end{lem}
For more information about flags, we refer the reader to \cite{sha} and \cite{har}.

\section{Upper triangular matrices and flags}
In this section, we explore the relation between upper triangular matrices and flags. First, we introduce some notation.
\begin{1def}
For a matrix $A \in {\rm Mat}_{d+1}(\mathbb{F})$, we define some submatrices of $A$ as follows. For $0\le i\le j \le d$, let $A[i,j]$ denote the submatrix $\{A_{kl}\}_{0\le k \le j-i, i\le l \le j}$ of $A$. Note that $A[0,d]=A$.
\end{1def}
\begin{1def}
For a matrix $A \in {\rm Mat}_{d+1}(\mathbb{F})$ and $0\le j\le d$, we call the submatrix $A[0,j]$ the {\it $j$-th leading principal submatrix} of $A$.
\end{1def}

\begin{1def}
For a matrix $A \in {\rm Mat}_{d+1}(\mathbb{F})$, we call it {\it good} whenever the submatrix $A[i,d]$ is invertible for $0 \le i \le d$.
\end{1def}
\begin{1def}
For a matrix $A \in {\rm Mat}_{d+1}(\mathbb{F})$, we call it {\it very good} whenever the submatrix $A[i,j]$ is invertible for $0 \le i \le j \le d$. 
\end{1def}

\begin{lem}
A matrix in ${\rm Mat}_{d+1}(\mathbb{F})$ is very good if and only if each of its leading principal submatrices is good.
\end{lem}
\begin{proof}
By Definitions 4.1--4.4.
\end{proof}
Referring to Definition 4.1, we now consider the case in which $A$ is upper triangular.
\begin{lem}
For an upper triangular matrix $A \in {\rm Mat}_{d+1}(\mathbb{F})$, the submatrix $A[0,j]$ is upper triangular for $0\le j \le d$.
\end{lem}
\begin{proof}
By Definition 4.1.
\end{proof}
\begin{lem}
For an invertible upper triangular matrix $A \in {\rm Mat}_{d+1}(\mathbb{F})$, the submatrix $A[0,j]$ is upper triangular and invertible for $0\le j \le d$.
\end{lem}
\begin{proof}
By Definition 4.1.
\end{proof}
\par
Consider an invertible upper triangular matrix $T\in {\rm Mat}_{d+1}(\mathbb{F})$. View $T$ as the transition matrix from a basis $\{u_i\}_{i=0}^{d}$ of $V$ to a basis $\{v_i\}_{i=0}^{d}$ of $V$. Thus for $0\le j\le d$,
\begin{equation}
v_j=\sum_{i=0}^{d}T_{ij}u_{i}.
\end{equation}

For the moment, pick $x\in V$. Then there exist scalars $\{b_i(x)\}_{i=0}^d$ in $\mathbb{F}$ such that
\begin{equation}
x=\sum_{i=0}^db_i(x)u_i.
\end{equation}
Moreover, there exist scalars $\{c_i(x)\}_{i=0}^d$ in $\mathbb{F}$ such that
\begin{equation}
x=\sum_{i=0}^dc_i(x)v_i.
\end{equation}
By (4.1)--(4.3),
\begin{equation}
Tc=b,
\end{equation}
where $c=(c_0(x),c_1(x),\dots,c_d(x))^t$ and $b=(b_0(x),b_1(x),\dots,b_d(x))^t$.

We now use $T$ to construct three flags on $V$.
\begin{lem}
With the above notation, the following two flags on $V$ coincide:
\begin{enumerate}
\item[(\rm i)] the flag induced by $\{u_i\}_{i=0}^{d}$;
\item[(\rm ii)] the flag induced by $\{v_i\}_{i=0}^{d}$.
\end{enumerate}
\end{lem}
\begin{proof}
By (4.1) and since $T$ is upper triangular.
\end{proof}
\par
We now define three flags on $V$, denoted by $\{U_{i}\}_{i=0}^{d}$, $\{U_{i}^{\prime}\}_{i=0}^{d}$, $\{U_{i}^{\prime\prime}\}_{i=0}^{d}$. The flag $\{U_{i}\}_{i=0}^{d}$ is induced by the basis $\{u_i\}_{i=0}^{d}$ or $\{v_i\}_{i=0}^{d}$. The flag $\{U_{i}^{\prime}\}_{i=0}^{d}$ (resp. $\{U_{i}^{\prime\prime}\}_{i=0}^{d}$) is induced by the basis $\{u_{d-i}\}_{i=0}^{d}$ (resp. $\{v_{d-i}\}_{i=0}^{d}$). More explicitly, for $0\le i\le d$,
\begin{equation}
U_{i}=\mathbb{F}u_0+\mathbb{F}u_1+\cdots +\mathbb{F}u_i=\mathbb{F}v_0+\mathbb{F}v_1+\cdots +\mathbb{F}v_i;
\end{equation}
\begin{equation}
U_{i}^{\prime}=\mathbb{F}u_d+\mathbb{F}u_{d-1}+\cdots +\mathbb{F}u_{d-i};
\end{equation}
\begin{equation}
U_{i}^{\prime\prime}=\mathbb{F}v_d+\mathbb{F}v_{d-1}+\cdots +\mathbb{F}v_{d-i}.
\end{equation}
\par
By Lemma 3.9, the flag $\{U_{i}\}_{i=0}^{d}$ is opposite to the flags $\{U_{i}^{\prime}\}_{i=0}^{d}$ and $\{U_{i}^{\prime\prime}\}_{i=0}^{d}$. Our next goal is to give a necessary and sufficient condition for the flags $\{U_{i}^{\prime}\}_{i=0}^{d}$ and $\{U_{i}^{\prime\prime}\}_{i=0}^{d}$ to be opposite. We will use the following lemma.
\begin{lem}
With the above notation, for $0\le i\le d-1$, $U_i^{\prime}\cap U_{d-1-i}^{\prime\prime}=0$ if and only if $\det(T[i+1,d])\ne 0$.
\end{lem}
\begin{proof}
Consider $x\in V$. We refer to the notation around (4.2) and (4.3). We make two observations about $x$.
The first observation is that by (4.6), we have $x\in U_i^{\prime}$ if and only if $b_n(x)=0$ for $0\le n \le d-1-i$. The second observation is that by (4.7), we have $x\in U_{d-1-i}^{\prime\prime}$ if and only if $c_n(x)=0$ for $0\le n \le i$. In this case, by (4.4),
\begin{equation}
T[i+1,d](c_{i+1}(x),c_{i+2}(x),\dots ,c_d(x))^{t}=(b_{0}(x),b_{1}(x),\dots ,b_{d-1-i}(x))^{t}.
\end{equation}

First assume that $\det(T[i+1,d])\ne 0$. We will show that $U_i^{\prime}\cap U_{d-1-i}^{\prime\prime}=0$. To do this, we assume $x\in U_i^{\prime}\cap U_{d-1-i}^{\prime\prime}$, and show that $x=0$. By the first observation, $b_n(x)=0$ for $0\le n \le d-1-i$. By the second observation, $c_n(x)=0$ for $0\le n \le i$ and (4.8) holds. By these comments,
\begin{equation}
T[i+1,d](c_{i+1}(x),c_{i+2}(x),\dots ,c_d(x))^{t}=0.
\end{equation}
By (4.9) and $\det(T[i+1,d])\ne 0$, the vector $(c_{i+1}(x),c_{i+2}(x),\dots ,c_d(x))^{t}=0$. In other words, $c_n(x)=0$ for $i+1\le n \le d$. We have shown that $c_n(x)=0$ for $0\le n\le d$. Hence $x=0$. Therefore $U_i^{\prime}\cap U_{d-1-i}^{\prime\prime}=0$.

Next assume that $\det(T[i+1,d])= 0$. We will show that $U_i^{\prime}\cap U_{d-1-i}^{\prime\prime}\ne 0$. By the assumption and linear algebra, there exists a nonzero vector
\begin{equation*}
w=(w_{i+1},w_{i+2},\dots ,w_{d})\in \mathbb{F}^{d-i}
\end{equation*}
such that $T[i+1,d]w^{t}=0$. Choose the vector $x$ such that $c_n(x)=0$ for $0\le n\le i$ and $c_n(x)=w_{n}$ for $i+1\le n \le d$. Observe that $x\ne 0$ and satisfies (4.9). By the second observation, $x \in U_{d-1-i}^{\prime\prime}$ and (4.8) holds. By (4.8) and (4.9), $b_n(x)=0$ for $0\le n \le d-1-i$. By the first observation, $x\in U_i^{\prime}$. We have shown that $0 \ne x \in U_i^{\prime}\cap U_{d-1-i}^{\prime\prime}$. Therefore $U_i^{\prime}\cap U_{d-1-i}^{\prime\prime}\ne 0$.
\end{proof}

\begin{pro}
The flags $\{U_{i}^{\prime}\}_{i=0}^{d}$ and $\{U_{i}^{\prime\prime}\}_{i=0}^{d}$ are opposite if and only if $T$ is good in the sense of Definition 4.3.
\end{pro}
\begin{proof}
Recall from Definition 4.1 that $T[0,d]=T$. Since $T$ is invertible, we obtain $\det(T[0,d]) \ne 0$. Therefore, by Definition 4.3, $T$ is good if and only if $\det(T[i+1,d])\ne 0$ for $0 \le i \le d-1$. By Lemma 4.9, this happens if and only if $U_i^{\prime}\cap U_{d-1-i}^{\prime\prime}=0$ for $0\le i \le d-1$. The result follows in view of Lemma 3.11.
\end{proof}
\begin{cor}
$T$ is good if and only if $T^{-1}$ is good.
\end{cor}
\begin{proof}
Going through the construction around (4.5)--(4.7) using $T$, we obtain a sequence of three flags $\{U_{i}\}_{i=0}^{d}$, $\{U_{i}^{\prime}\}_{i=0}^{d}$, $\{U_{i}^{\prime\prime}\}_{i=0}^{d}$. Repeating the construction with $T$ replaced by $T^{-1}$, we obtain a sequence of three flags $\{U_{i}\}_{i=0}^{d}$, $\{U_{i}^{\prime\prime}\}_{i=0}^{d}$, $\{U_{i}^{\prime}\}_{i=0}^{d}$. The result follows in view of Proposition 4.10.
\end{proof}
Next, we give a necessary and sufficient condition for the three flags $\{U_{i}\}_{i=0}^{d}$, $\{U_{i}^{\prime}\}_{i=0}^{d}$, $\{U_{i}^{\prime\prime}\}_{i=0}^{d}$ to be totally opposite in the sense of Definition 3.12.
\begin{pro}
The three flags $\{U_{i}\}_{i=0}^{d}$, $\{U_{i}^{\prime}\}_{i=0}^{d}$, $\{U_{i}^{\prime\prime}\}_{i=0}^{d}$ are totally opposite in the sense of Definition 3.12 if and only if $T$ is very good in the sense of Definition 4.4.
\end{pro}
\begin{proof}
By parts (i), (ii) of Lemma 3.13 along with Definition 4.4 and Lemma 4.5, it suffices to show that for $0\le n\le d$, the sequences $\{U_{i}\}_{i=0}^{d-n}$, $\{U_{d-n}\cap U_{n+i}^{\prime}\}_{i=0}^{d-n}$, $\{U_{d-n}\cap U_{n+i}^{\prime\prime}\}_{i=0}^{d-n}$ are mutually opposite flags on $U_{d-n}$ if and only if $\det(T[i,d-n])\ne 0$ for $0\le i  \le d-n$. Let $n$ be given.
By (4.5)--(4.7) we find that for $0\le i \le d-n$,
\begin{equation}
U_{d-n}\cap U_{n+i}^{\prime}=\mathbb{F}u_{d-n-i}+\mathbb{F}u_{d-n-i+1}+\dots+\mathbb{F}u_{d-n};
\end{equation}
\begin{equation}
U_{d-n}\cap U_{n+i}^{\prime\prime}=\mathbb{F}v_{d-n-i}+\mathbb{F}v_{d-n-i+1}+\dots+\mathbb{F}v_{d-n}.
\end{equation}
By (4.5), (4.10) and Lemma 3.9, the flag $\{U_{i}\}_{i=0}^{d-n}$ on $U_{d-n}$ is opposite to the flag $\{U_{d-n}\cap U_{n+i}^{\prime}\}_{i=0}^{d-n}$ on $U_{d-n}$. Similarly, by (4.5), (4.11) and Lemma 3.9, the flag $\{U_{i}\}_{i=0}^{d-n}$ on $U_{d-n}$ is opposite to the flag  $\{U_{d-n}\cap U_{n+i}^{\prime\prime}\}_{i=0}^{d-n}$ on $U_{d-n}$.
By Lemma 4.7, the submatrix $T[0,d-n]$ is invertible and upper triangular. Therefore we can apply Proposition 4.10 to the two flags $\{U_{d-n}\cap U_{n+i}^{\prime}\}_{i=0}^{d-n}$ and $\{U_{d-n}\cap U_{n+i}^{\prime\prime}\}_{i=0}^{d-n}$ on $U_{d-n}$. By this, the two flags $\{U_{d-n}\cap U_{n+i}^{\prime}\}_{i=0}^{d-n}$ and $\{U_{d-n}\cap U_{n+i}^{\prime\prime}\}_{i=0}^{d-n}$ are opposite if and only if $\det(T[i,d-n])\ne 0$ for $0\le i\le d-n$. We have shown that the sequences $\{U_{i}\}_{i=0}^{d-n}$, $\{U_{d-n}\cap U_{n+i}^{\prime}\}_{i=0}^{d-n}$, $\{U_{d-n}\cap U_{n+i}^{\prime\prime}\}_{i=0}^{d-n}$ are mutually opposite flags on $U_{d-n}$ if and only if $\det(T[i,d-n])\ne 0$ for $0\le i  \le d-n$. The result follows.
\end{proof}
\begin{cor}
$T$ is very good if and only if $T^{-1}$ is very good.
\end{cor}
\begin{proof}
Similar to the proof of Corollary 4.11.
\end{proof}
\begin{1def}
Let $\mathcal{T}_d(\mathbb{F})$ denote the set of very good upper triangular matrices in ${\rm Mat}_{d+1}(\mathbb{F})$. Note that each element of $\mathcal{T}_d(\mathbb{F})$ is invertible.
\end{1def}

\section{Billiard Arrays}
In this section, we develop some results about Billiard Arrays that will be used later in the paper. We will refer to the following basis for the vector space $\mathbb{R}^3$:
\begin{equation*}
e_1=(1,0,0),\qquad e_2=(0,1,0),\qquad e_3=(0,0,1).
\end{equation*}
Define a subset $\Phi = \lbrace e_i - e_j | 1 \leq i,j\leq 3, i\not=j\rbrace$ of $\mathbb{R}^3$. The set $\Phi$ is often called the root system $A_2$. For notational convenience define
\begin{equation}
\alpha=e_1-e_2, \qquad \beta=e_2-e_3, \qquad \gamma=e_3-e_1.
\end{equation}
Note that
\begin{equation*}
\Phi=\{\pm\alpha,\pm\beta,\pm\gamma\}, \qquad \alpha+\beta+\gamma=0.
\end{equation*}

\begin{1def}
Let $\bigtriangleup_{d}$ denote the subset of $\mathbb{R}^3$ consisting of the three-tuples of natural numbers whose sum is $d$. Thus
\begin{equation*}
\bigtriangleup_{d}=\{(r,s,t)\mid r,s,t \in \mathbb{N}, r+s+t=d\}.
\end{equation*}
\end{1def}

\begin{rmk}
We can arrange the elements of $\bigtriangleup_{d}$ in a triangular array. For example, if $d=3$, the array looks as follows after deleting all punctuation:
\[
    \setlength{\arraycolsep}{1pt}
    \begin{array}{ccccccc}
        &   &   & 030 &   &   &   \\
        &   & 120 &   & 021 &   &   \\
        & 210 &   & 111 &   & 012 &   \\
      300 &   & 201 &   & 102 &   & 003 \\
    \end{array}
  \]
An element in $\bigtriangleup_{d}$ is called a {\it location}.
\end{rmk}
\begin{1def}
For $\eta\in \{1,2,3\}$, the $\eta$-{\it corner} of $\bigtriangleup_{d}$ is the location in $\bigtriangleup_{d}$ that has $\eta$-coordinate $d$ and all other coordinates 0. By a {\it corner} of $\bigtriangleup_{d}$ we mean the 1-corner or 2-corner or 3-corner. The corners in $\bigtriangleup_{d}$ are listed below.
\begin{displaymath}
de_1=(d,0,0), \qquad de_2=(0,d,0), \qquad de_3=(0,0,d).
\end{displaymath}
\end{1def}
\begin{1def}
For $\eta\in \{1,2,3\}$, the $\eta$-{\it boundary} of $\bigtriangleup_{d}$ is the set of locations in $\bigtriangleup_{d}$ that have $\eta$-coordinate $0$. The {\it boundary} of $\bigtriangleup_{d}$ is the union of its 1-boundary, 2-boundary and 3-boundary. By the {\it interior} of $\bigtriangleup_{d}$ we mean the set of locations in $\bigtriangleup_{d}$ that are not on the boundary.
\end{1def}
\begin{exa}
Referring to the picture of $\bigtriangleup_{3}$ from Remark 5.2, the 2-boundary of $\bigtriangleup_{3}$ consists of the four locations in the bottom row.
\end{exa}
\begin{1def}
For $\eta\in \{1,2,3\}$ we define a binary relation on $\bigtriangleup_{d}$ called $\eta$-{\it collinearity}. By definition, locations $\lambda,\lambda^{\prime}$ in $\bigtriangleup_{d}$ are $\eta$-collinear whenever the $\eta$-coordinate of $\lambda-\lambda^{\prime}$ is 0. Note that $\eta$-collinearity is an equivalence relation. Each equivalence class will be called an {\it $\eta$-line}. By a {\it line} in $\bigtriangleup_{d}$ we mean a 1-line or 2-line or 3-line.
\end{1def}
\begin{exa}
Referring to the picture of $\bigtriangleup_{3}$ from Remark 5.2, the horizontal rows are the $2$-lines of $\bigtriangleup_{3}$.
\end{exa}

\begin{1def}
Locations $\lambda,\mu$ in $\bigtriangleup_{d}$ are called {\it adjacent} whenever $\lambda-\mu\in \Phi$.
\end{1def}

\begin{1def}
By an {\it edge} in $\bigtriangleup_{d}$ we mean a set of two adjacent locations.
\end{1def}

\begin{1def}
By a {\it 3-clique} in $\bigtriangleup_{d}$ we mean a set of three mutually adjacent locations in $\bigtriangleup_{d}$. There are two kinds of 3-cliques: $\bigtriangleup$({\it black}) and $\bigtriangledown$({\it white}).
\end{1def}

\begin{lem}
\cite[Lemma~4.31]{paul} Assume $d\ge 1$. We describe a bijection from $\bigtriangleup_{d-1}$ to the set of black 3-cliques in $\bigtriangleup_{d}$. The bijection sends each $(r,s,t)\in \bigtriangleup_{d-1}$ to the black 3-clique in $\bigtriangleup_{d}$ consisting of the locations $(r+1,s,t),(r,s+1,t), (r,s,t+1)$.
\end{lem}

\begin{lem}
\cite[Lemma~4.32]{paul} Assume $d\ge 2$. We describe a bijection from $\bigtriangleup_{d-2}$ to the set of white 3-cliques in $\bigtriangleup_{d}$. The bijection sends each $(r,s,t)\in \bigtriangleup_{d-2}$ to the white 3-clique in $\bigtriangleup_{d}$ consisting of the locations $(r,s+1,t+1),(r+1,s,t+1), (r+1,s+1,t)$.
\end{lem}

\begin{lem}
\cite[Lemma~4.33]{paul} For $\bigtriangleup_{d}$, each edge is contained in a unique black 3-clique and at most one white 3-clique.
\end{lem}

Let $\mathcal{P}_1(V)$ denote the set of $1$-dimensional subspaces of $V$.
\begin{1def}
\cite[Definition~7.1]{paul} By a {\it Billiard Array} on $V$ we mean a function $B: \bigtriangleup_{d}\to \mathcal{P}_{1}(V), \lambda \mapsto B_{\lambda}$ that satisfies the following conditions:
\begin{enumerate}
\item[\rm(i)] for each line $L$ in $\bigtriangleup_{d}$ the sum $\sum_{\lambda\in L}B_{\lambda}$ is direct;
\item[\rm(ii)] for each black 3-clique $C$ in $\bigtriangleup_{d}$ the sum $\sum_{\lambda\in C}B_{\lambda}$ is not direct.
\end{enumerate}
We say that $B$ is {\it over} $\mathbb{F}$. We call $V$ the {\it underlying vector space}. We call $d$ the {\it diameter} of $B$.
\end{1def}
\begin{lem}
\cite[Corollary~7.4]{paul} Let $B$ denote a Billiard Array on $V$. Let $\lambda,\mu,\nu$ denote the locations in $\bigtriangleup_{d}$ that form a black 3-clique. Then each of $B_\lambda,B_\mu,B_\nu$ is contained in the sum of the other two.
\end{lem}
\begin{1def}
Let $V^{\prime}$ denote a vector space over $\mathbb{F}$ with dimension d+1. Let $B$ (resp. $B^{\prime}$) denote a Billiard Array on $V$ (resp. $V^\prime$). By an {\it isomorphsim} of Billiard Arrays from $B$ to $B^\prime$ we mean an $\mathbb{F}$-linear bijection $\sigma:  V \to V^\prime$ that sends $B_\lambda \mapsto B^{\prime}_{\lambda}$ for all $\lambda \in \bigtriangleup_{d}$. The Billiard Arrays $B$ and $B^\prime$ are called {\it isomorphic} whenever there exists an isomorphism of Billiard Arrays from $B$ to $B^\prime$.
\end{1def}
From now until the end of Lemma 5.21, let $B$ denote a Billiard Array on $V$.

\begin{1def}
Pick $\eta\in\{1, 2, 3\}$. Following \cite[Section~9]{paul} we now define a flag on $V$ called the {\it$B$-flag $[\eta]$}. For $0\le i\le d$, the $i$-component of this flag is $\sum_{\lambda}B_\lambda$, where the sum is over all $\lambda\in \bigtriangleup_{d}$ that have $\eta$-coordinate at least $d-i$.
\end{1def}

\begin{1def}
Pick distinct $\eta,\xi\in\{1, 2, 3\}$. Following \cite[Section~10]{paul} we now define a decomposition of $V$ called the {\it $B$-decomposition $[\eta,\xi]$}. For $0 \leq i \leq d$ the $i$-component of this decomposition is the subspace $B_\lambda$, where the location $\lambda$ is described in the table below:
\bigskip

\centerline{
\begin{tabular}[t]{c|c|c}
 $\eta $ & $\xi$ &
    $\lambda_i$ 
   \\ \hline  \hline
$1$ & $2$ &
$(d-i,i,0)$ 
  \\ 
  \hline
$2$ & $1$ &
 $(i,d-i,0)$
   \\
\hline
$2$ & $3$ &
 $(0,d-i,i)$ 
  \\
  \hline
$3$ & $2$ &
  $(0,i,d-i)$
   \\
\hline
$3$ & $1$ &
$(i,0,d-i)$ 
  \\
  \hline
$1$ & $3$ &
 $(d-i,0,i)$
   \\
     \end{tabular}}
     \bigskip
\end{1def}

\begin{lem}
\cite[Lemma~10.6]{paul} For distinct $\eta, \xi\in\{1, 2, 3\}$ the $B$-decomposition $[\eta,\xi]$ of $V$ induces the $B$-flag $[\eta]$ on $V$.
\end{lem}

\begin{lem}
\cite[Theorem~12.4]{paul} The $B$-flags $[1]$, $[2]$, $[3]$ on $V$ from Definition 5.17 are totally opposite in the sense of Definition 3.12.
\end{lem}

\begin{lem}
\cite[Corollary~11.2]{paul} Pick a location $\lambda=(r,s,t)$ in $\bigtriangleup_{d}$. Then $B_\lambda$ is equal to the intersection of the following three sets:
\begin{enumerate}
\item[\rm(i)] component $d-r$ of the $B$-flag $[1]$;
\item[\rm(ii)] component $d-s$ of the $B$-flag $[2]$;
\item[\rm(iii)] component $d-t$ of the $B$-flag $[3]$.
\end{enumerate}
\end{lem}

\begin{thm}
\cite[Theorem~12.7]{paul} Suppose that we are given three totally opposite flags on $V$, denoted by $\{W_{i}\}_{i=0}^{d}$, $\{W_{i}^{\prime}\}_{i=0}^{d}$, $\{W_{i}^{\prime\prime}\}_{i=0}^{d}$. For each location $\lambda=(r,s,t)$ in $\bigtriangleup_{d}$, define $B_{\lambda}=W_{d-r}\cap W_{d-s}^{\prime}\cap W_{d-t}^{\prime\prime}$. Then the map $B: \bigtriangleup_{d}\to \mathcal{P}_{1}(V), \lambda \mapsto B_{\lambda}$ is a Billiard Array on $V$.
\end{thm}

\begin{1def}
\cite[Definition~8.1]{paul} By a {\it Concrete Billiard Array} on $V$ we mean a function $\mathcal{B}: \bigtriangleup_{d}\to V, \lambda \mapsto \mathcal{B}_{\lambda}$ that satisfies the following conditions:
\begin{enumerate}
\item[\rm(i)] for each line $L$ in $\bigtriangleup_{d}$ the vectors $\{\mathcal{B}_{\lambda}\}_{\lambda\in L}$ are linearly independent;
\item[\rm(ii)] for each black 3-clique $C$ in $\bigtriangleup_{d}$ the vectors $\{\mathcal{B}_{\lambda}\}_{\lambda\in C}$ are linearly dependent.
\end{enumerate}
We say that $\mathcal{B}$ is {\it over} $\mathbb{F}$. We call $V$ the {\it underlying vector space}. We call $d$ the {\it diameter} of $\mathcal{B}$.
\end{1def}
\begin{exa}
Let $B$ denote a Billiard Array on $V$. For $\lambda\in\bigtriangleup_{d}$ pick $0\ne \mathcal{B}_\lambda\in B_\lambda$. Then the function $\mathcal{B}: \bigtriangleup_{d}\to V, \lambda \mapsto \mathcal{B}_{\lambda}$ is a Concrete Billiard Array on $V$.
\end{exa}
\begin{1def}
Let $B$ denote a Billiard Array on $V$, and let $\mathcal{B}$ denote a concrete Billiard Array on $V$. We say that $B$, $\mathcal{B}$ {\it correspond} whenever $\mathcal{B}_{\lambda}$ spans $B_{\lambda}$ for all $\lambda\in\bigtriangleup_{d}$.
\end{1def}
\begin{1def}
Let $\lambda, \mu$ denote locations in $\bigtriangleup_{d}$ that form an edge. By Lemma 5.13 there exists a unique location $\nu\in \bigtriangleup_{d}$ such that $\lambda, \mu, \nu$ form a black 3-clique. We call $\nu$ the {\it completion} of the edge.
\end{1def}
From now until the end of Definition 5.33, let $B$ denote a Billiard Array on $V$.
\begin{1def}
Let $\lambda, \mu$ denote locations in $\bigtriangleup_{d}$ that form an edge. By a {\it brace} for this edge, we mean a set of nonzero vectors $u\in B_{\lambda}, v\in B_{\mu}$ such that $u+v\in B_{\nu}$. Here $\nu$ denotes the completion of the edge.
\end{1def}

\begin{lem}
\cite[Lemma~13.12]{paul} Let $\lambda, \mu$ denote locations in $\bigtriangleup_{d}$ that form an edge. Then each nonzero $u\in B_{\lambda}$ is contained in a unique brace for this edge.
\end{lem}

\begin{1def}
\cite[Definition~14.1]{paul} Let $\lambda, \mu$ denote adjacent locations in $\bigtriangleup_{d}$. We define an $\mathbb{F}$-linear map $\widetilde{B}_{\lambda,\mu}: B_{\lambda}\to B_{\mu}$ as follows. For each brace $u\in B_{\lambda}, v\in B_{\mu}$, the map $\widetilde{B}_{\lambda,\mu}$ sends $u\mapsto v$. The map $\widetilde{B}_{\lambda,\mu}$ is well defined by Lemma 5.28. Observe that $\widetilde{B}_{\lambda,\mu}: B_{\lambda}\to B_{\mu}$ is bijective.
\end{1def}
\begin{1def}
\cite[Definition~14.9]{paul} Let $\lambda, \mu, \nu$ denote locations in $\bigtriangleup_{d}$ that form a white 3-clique. Then the composition
\begin{equation*}
B_{\lambda}\stackrel{\widetilde{B}_{\lambda,\mu}}{\longrightarrow}B_{\mu}\stackrel{\widetilde{B}_{\mu,\nu}}{\longrightarrow}B_{\nu}\stackrel{\widetilde{B}_{\nu,\lambda}}{\longrightarrow}B_{\lambda}
\end{equation*}
is a nonzero scalar multiple of the identity map on $B_{\lambda}$. The scalar is called the {\it clockwise $B$-value} (resp. {\it counterclockwise $B$-value}) of the clique whenever $\lambda, \mu, \nu$ runs clockwise (resp. counterclockwise) around the clique.
\end{1def}

\begin{1def}
For each white 3-clique in $\bigtriangleup_{d}$, by its {\it $B$-value} we mean its clockwise $B$-value.
\end{1def}
\begin{1def}
By a {\it value function} on $\bigtriangleup_{d}$, we mean a function $\psi:  \bigtriangleup_{d} \to \mathbb{F}\setminus\{0\}$.
\end{1def}
\begin{1def}
\cite[Definition~14.13]{paul} Assume $d\ge 2$. We define a function $\widehat{B}:  \bigtriangleup_{d-2} \to \mathbb{F}$ as follows. Pick $(r,s,t)\in \bigtriangleup_{d-2}$. To describe the image of $(r,s,t)$ under $\widehat{B}$, consider the corresponding white 3-clique in $\bigtriangleup_{d}$ from Lemma 5.12. The $B$-value of this 3-clique is the image of $(r,s,t)$ under $\widehat{B}$. Observe that $\widehat{B}$ is a value function on $\bigtriangleup_{d-2}$ in the sense of Definition 5.32. We call $\widehat{B}$ the {\it value function} for B.
\end{1def}
\begin{1def}
Let $BA_d(\mathbb{F})$ denote the set of isomorphism classes of Billiard Arrays over $\mathbb{F}$ that have diameter $d$.
\end{1def}
\begin{1def}
Let $VF_d(\mathbb{F})$ denote the set of value functions on $\bigtriangleup_{d}$.
\end{1def}
\begin{1def}
Assume $d\ge 2$. We now define a map $\theta:BA_d(\mathbb{F})\to VF_{d-2}(\mathbb{F})$. For $B\in BA_d(\mathbb{F})$, the image of $B$ under $\theta$ is the value function $\widehat{B}$ from Definition 5.33.
\end{1def}
\begin{lem}
\cite[Lemma~19.1]{paul} Assume $d\ge 2$. Then the map $\theta:BA_d(\mathbb{F})\to VF_{d-2}(\mathbb{F})$ from Definition 5.36 is bijective.
\end{lem}
\begin{1def}
Let $\mathcal{B}$ denote a Concrete Billiard Array on $V$. Let $B$ denote the corresponding Billiard Array on $V$ from Definition 5.25. Let $\lambda, \mu$ denote adjacent locations in $\bigtriangleup_{d}$. Recall the bijection $\widetilde{B}_{\lambda,\mu}: B_{\lambda}\to B_{\mu}$ from Definition 5.29. Recall that $\mathcal{B}_{\lambda}$ is a basis for $B_{\lambda}$ and $\mathcal{B}_{\mu}$ is a basis for $B_{\mu}$. Define a scalar $\widetilde{\mathcal{B}}_{\lambda,\mu}\in \mathbb{F}$ such that $\widetilde{B}_{\lambda,\mu}$ sends $\mathcal{B}_{\lambda}\mapsto \widetilde{\mathcal{B}}_{\lambda,\mu}\mathcal{B}_{\mu}$. Note that $\widetilde{\mathcal{B}}_{\lambda,\mu}\ne 0$.
\end{1def}
\begin{lem}
\cite[Lemma~15.6]{paul} Let $\mathcal{B}$ denote a Concrete Billiard Array on $V$. Let $\lambda, \mu, \nu$ denote locations in $\bigtriangleup_{d}$ that form a black 3-clique. Then
\begin{equation*}
\mathcal{B}_{\lambda}+\widetilde{\mathcal{B}}_{\lambda,\mu}\mathcal{B}_{\mu}+\widetilde{\mathcal{B}}_{\lambda,\nu}\mathcal{B}_{\nu}=0.
\end{equation*}
\end{lem}
\begin{lem}
\cite[Lemma~15.9]{paul} Let $\mathcal{B}$ denote a Concrete Billiard Array on $V$. Let $B$ denote the corresponding Billiard Array on $V$ from Definition 5.25. Let $\lambda, \mu, \nu$ denote the locations in $\bigtriangleup_{d}$ that form a white 3-clique. Then the clockwise (resp. counterclockwise) $B$-value of the clique is equal to
\begin{equation*}
\widetilde{\mathcal{B}}_{\lambda,\mu}\widetilde{\mathcal{B}}_{\mu,\nu}\widetilde{\mathcal{B}}_{\nu, \lambda}
\end{equation*}
whenever the sequence $\lambda, \mu, \nu$ runs clockwise (resp. counterclockwise) around the clique.
\end{lem}
Next we consider the $2$-boundary of $\bigtriangleup_{d}$.
\begin{pro}
Given a Billiard Array $B$ on $V$, let $\{V_i\}_{i=0}^d$ denote the $B$-decomposition $[1,3]$ of $V$ from Definition 5.18. Then for $\lambda=(r,s,t)\in \bigtriangleup_{d}$,
\begin{equation}
B_{\lambda}\subseteq V_t+V_{t+1}+\dots+V_{d-r}.
\end{equation}
\end{pro}
\begin{proof}
We do induction on $s$. First assume that $s=0$. Then by Definition 5.18, $B_{\lambda}=V_t$. Next assume that $s>0$. Consider the black 3-clique in $\bigtriangleup_{d}$ with the locations $\lambda=(r,s,t), \mu=(r,s-1,t+1), \nu=(r+1,s-1,t)$. By induction,
\begin{eqnarray}
{B}_{\mu}\subseteq V_{t+1}+V_{t+2}+\dots+V_{d-r};\\
{B}_{\nu}\subseteq V_{t}+V_{t+1}+\dots+V_{d-r-1}.
\end{eqnarray}
By Lemma 5.15,
\begin{equation}
B_{\lambda}\subseteq {B}_{\mu}+{B}_{\nu}.
\end{equation}
The equation (5.2) follows from (5.3)--(5.5).
\end{proof}
\begin{lem}
Let $B$ denote a Billiard Array on $V$. Let $\{u_i\}_{i=0}^d$ $($resp. $\{v_i\}_{i=0}^d$$)$ denote a basis of $V$ that induces the $B$-decomposition $[1,2]$ $($resp. $B$-decomposition $[1,3]$$)$. Then the transition matrices between $\{u_i\}_{i=0}^d$ and $\{v_i\}_{i=0}^d$ are upper triangular.
\end{lem}
\begin{proof}
By Proposition 5.41.
\end{proof}
We next show that the transition matrices in Lemma 5.42 are very good in the sense of Definition 4.4. By Corollary 4.13, it suffices to show that the transition matrix from $\{u_i\}_{i=0}^d$ to $\{v_i\}_{i=0}^d$ is very good.
\begin{lem}
Referring to Lemma 5.42, let $T$ denote the transition matrix from $\{u_i\}_{i=0}^d$ to $\{v_i\}_{i=0}^d$. Then $T$ is very good in the sense of Definition 4.4.
\end{lem}
\begin{proof}
Consider the corresponding three flags $\{U_{i}\}_{i=0}^{d}$, $\{U_{i}^{\prime}\}_{i=0}^{d}$, $\{U_{i}^{\prime\prime}\}_{i=0}^{d}$ on $V$ from (4.5)--(4.7). By Lemma 5.19, the $B$-flag $[1]$ (resp. $[2]$) (resp. $[3]$) is the flag $\{U_i\}_{i=0}^d$ (resp. $\{U_i^\prime\}_{i=0}^d$) (resp. $\{U_i^{\prime\prime}\}_{i=0}^d$). By Lemma 5.20, the three flags $\{U_{i}\}_{i=0}^{d}$, $\{U_{i}^{\prime}\}_{i=0}^{d}$, $\{U_{i}^{\prime\prime}\}_{i=0}^{d}$ are totally opposite. By Proposition 4.12, $T$ is very good.
\end{proof}
Recall the set $\mathcal{T}_d(\mathbb{F})$ from Definition 4.14.
\begin{cor}
Referring to Lemma 5.43, $T\in \mathcal{T}_d(\mathbb{F})$.
\end{cor}
\begin{proof}
By Lemmas 5.42, 5.43.
\end{proof}
\begin{1def}
For location $\tau=(r,s,t)\in \bigtriangleup_{d-1}$, consider the corresponding black 3-clique in $\bigtriangleup_{d}$ from Lemma 5.11, with locations $\lambda=(r+1,s,t), \mu=(r,s+1,t), \nu=(r,s,t+1)$. Given a Concrete Billiard Array $\mathcal{B}$ on $V$, we say that $\mathcal{B}$ is {\it $\tau$-standard} whenever $\mathcal{B}_{\lambda}-\mathcal{B}_{\mu}\in \mathbb{F}\mathcal{B}_{\nu}$. We call $\mathcal{B}$ {\it standard} whenever $\mathcal{B}$ is $\tau$-standard for all $\tau\in \bigtriangleup_{d-1}$.
\end{1def}
Let $B$ denote a Billiard Array on $V$. For $0\le i \le d$, pick $0\ne f_i\in B_\kappa$ where $\kappa=(d-i,0,i)\in \bigtriangleup_{d}$. Observe that $\{f_i\}_{i=0}^{d}$ is a basis of $V$ that induces the $B$-decomposition $[1,3]$.
\begin{lem}
With the above notation, there exists a unique standard Concrete Billiard Array $\mathcal{B}$ on $V$ such that
\begin{enumerate}
\item[\rm(i)] $\mathcal{B}$ corresponds to $B$ in the sense of Definition 5.25;
\item[\rm(ii)] for $0\le i \le d$, $\mathcal{B}_{\kappa}=f_i$ where $\kappa=(d-i,0,i)\in \bigtriangleup_{d}$.
\end{enumerate}
\end{lem}
\begin{proof}
First we show that $\mathcal{B}$ exists. Let $\lambda=(r,s,t)\in \bigtriangleup_{d}$. We construct $\mathcal{B}_{\lambda}$ by induction on $s$. For $s=0$ define $\mathcal{B}_{\lambda}=f_t$. By the construction $0\ne\mathcal{B}_{\lambda}\in B_{\lambda}$. Next assume that $s>0$. Consider the black 3-clique in $\bigtriangleup_{d}$ with the locations $\lambda=(r,s,t), \mu=(r,s-1,t+1), \nu=(r+1,s-1,t)$. The vectors $\mathcal{B}_{\mu}$ and $\mathcal{B}_{\nu}$ have been determined by the induction. By Lemma 5.28, there exists a nonzero vector $\mathcal{B}_{\lambda}\in B_{\lambda}$ such that $\mathcal{B}_{\lambda}-\mathcal{B}_{\nu}\in \mathbb{F}\mathcal{B}_{\mu}$. We have constructed $\mathcal{B}_{\lambda}$ such that $0\ne \mathcal{B}_{\lambda}\in B_{\lambda}$ for all $\lambda\in \bigtriangleup_{d}$. By Example 5.24 and Definition 5.25, $\mathcal{B}$ is a Concrete Billiard Array on $V$ that corresponds to $B$. By the above construction and Definition 5.45, $\mathcal{B}$ is standard. We have shown that $\mathcal{B}$ exists. The uniqueness of $\mathcal{B}$ follows by (ii) and Definition 5.45.
\end{proof}
Consider the standard Concrete Billiard Array $\mathcal{B}$ from Lemma 5.46. For location $\lambda=(r,s,t)\in \bigtriangleup_{d}$, by Proposition 5.41 the vector $\mathcal{B}_{\lambda}$ is a linear combination of $f_t, f_{t+1},\dots ,f_{d-r}$. Let $\{a_i(\lambda)\}_{i=t}^{d-r}$ denote the corresponding coefficients, so that
\begin{equation}
\mathcal{B}_{\lambda}=a_t(\lambda)f_t+a_{t+1}(\lambda)f_{t+1}+\dots + a_{d-r}(\lambda)f_{d-r}.
\end{equation}
\begin{lem}
With the above notation, $a_t(\lambda)=1$.
\end{lem}
\begin{proof}
We do induction on $s$. First assume that $s=0$. Then by Lemma 5.46(ii), $\mathcal{B}_{\lambda}=f_{t}$. Therefore $a_t(\lambda)=1$. Next assume that $s>0$. Consider the black 3-clique in $\bigtriangleup_{d}$ with the locations $\lambda=(r,s,t), \mu=(r,s-1,t+1), \nu=(r+1,s-1,t)$. By Proposition 5.41,
\begin{equation}
\mathcal{B}_{\nu}=a_t(\nu)f_t+a_{t+1}(\nu)f_{t+1}+\dots + a_{d-r-1}(\nu)f_{d-r-1};
\end{equation}
\begin{equation}
\mathcal{B}_{\mu}=a_{t+1}(\mu)f_{t+1}+a_{t+2}(\mu)f_{t+2}+\dots + a_{d-r}(\mu)f_{d-r}.
\end{equation}
By induction,
\begin{equation}
a_t(\nu)=1.
\end{equation}
Since $\mathcal{B}$ is standard,
\begin{equation}
\mathcal{B}_{\lambda}-\mathcal{B}_{\nu}\in \mathbb{F}\mathcal{B}_{\mu}.
\end{equation}
By (5.6)--(5.10) we have $a_t(\lambda)=1$.
\end{proof}
For more information about Billiard Arrays, we refer the reader to \cite{paul}.

\section{Upper triangular matrices and Billiard Arrays}
Recall the set $\mathcal{T}_d(\mathbb{F})$ from Definition 4.14. In this section, we consider a matrix $T\in \mathcal{T}_d(\mathbb{F})$. Using $T$ we construct a Billiard Array $B$. Then for each white 3-clique in $\bigtriangleup_d$, we compute its $B$-value in terms of the entries of $T$.
\begin{1def}
Recall the set $BA_d(\mathbb{F})$ from Definition 5.34. We define a map $b:  \mathcal{T}_d(\mathbb{F}) \to BA_d(\mathbb{F})$ as follows. Let $T\in \mathcal{T}_d(\mathbb{F})$. View $T$ as the transition matrix from a basis $\{u_i\}_{i=0}^{d}$ of $V$ to a basis $\{v_i\}_{i=0}^{d}$ of $V$ as around (4.1). Consider the corresponding three flags $\{U_{i}\}_{i=0}^{d}$, $\{U_{i}^{\prime}\}_{i=0}^{d}$, $\{U_{i}^{\prime\prime}\}_{i=0}^{d}$ on $V$ from (4.5)--(4.7). These flags are totally opposite by Proposition 4.12, so they correspond to a Billiard Array on $V$ by Theorem 5.22. Since the bases $\{u_i\}_{i=0}^{d}$ and $\{v_i\}_{i=0}^{d}$ are not uniquely determined, this Billiard Array is only defined up to isomorphism of Billiard Arrays. The isomorphism class of this Billiard Array is the image of $T$ under $b$.
\end{1def}
In this section, we fix $T\in \mathcal{T}_d(\mathbb{F})$. By Definition 4.14, $T$ is upper triangular and invertible. Fix the two bases $\{u_i\}_{i=0}^{d}$, $\{v_i\}_{i=0}^{d}$ of $V$ as around (4.1) and the three flags $\{U_{i}\}_{i=0}^{d}$, $\{U_{i}^{\prime}\}_{i=0}^{d}$, $\{U_{i}^{\prime\prime}\}_{i=0}^{d}$ on $V$ from (4.5)--(4.7). Let $B$ denote the corresponding Billiard Array on $V$ from Theorem 5.22. Observe that $B\in b(T)$.
\begin{lem}
With the above notation, $B$ is the unique Billiard Array on $V$ such that for $0\le i \le d$,
\begin{enumerate}
\item[\rm(i)] $B_{\mu}=\mathbb{F}u_i$ where $\mu=(d-i,i,0)\in \bigtriangleup_{d}$;
\item[\rm(ii)] $B_{\nu}=\mathbb{F}v_i$ where $\nu=(d-i,0,i)\in \bigtriangleup_{d}$.
\end{enumerate}
\end{lem}
\begin{proof}
First we show that $B$ satisfies (i). By Theorem 5.22, $B_\mu=U_i\cap U^\prime_{d-i}\cap U^{\prime\prime}_{d}$. The subspace $U_i$ is given by (4.5). By (4.6), $U^\prime_{d-i}=\mathbb{F}u_d+\mathbb{F}u_{d-1}+\dots+\mathbb{F}u_i$. By (4.7), $U^{\prime\prime}_{d}=\mathbb{F}v_d+\mathbb{F}v_{d-1}+\dots+\mathbb{F}v_0=V$. By the above comments, $U_i\cap U^\prime_{d-i}\cap U^{\prime\prime}_{d}=\mathbb{F}u_i$. Hence $B_\mu=\mathbb{F}u_i$. Therefore $B$ satisfies (i). Similarly, $B$ satisfies (ii). Next we show that $B$ is the unique Billiard Array on $V$ that satisfies (i) and (ii). Suppose that $B^\prime$ is a Billiard Array on $V$ that satisfies (i) and (ii). By Lemma 5.19, the $B^\prime$-flag $[1]$ (resp. $[2]$) (resp. $[3]$) is the flag $\{U_i\}_{i=0}^d$ (resp. $\{U_i^\prime\}_{i=0}^d$) (resp. $\{U_i^{\prime\prime}\}_{i=0}^d$). By Lemma 5.21 and Definition 6.1, $B_\lambda=B^\prime_\lambda$ for all $\lambda\in \bigtriangleup_{d}$. Therefore $B=B^\prime$. We conclude that $B$ is the unique Billiard Array on $V$ that satisfies (i) and (ii).
\end{proof}
\begin{lem}
Consider the Billiard Array $B$ on $V$. There exists a unique standard Concrete Billiard Array $\mathcal{B}$ on $V$ such that
\begin{enumerate}
\item[\rm(i)] $\mathcal{B}$ corresponds to $B$ in the sense of Definition 5.25;
\item[\rm(ii)] for $0\le i \le d$, $\mathcal{B}_{\nu}=v_i$ where $\nu=(d-i,0,i)\in \bigtriangleup_{d}$.
\end{enumerate}
\end{lem}
\begin{proof}
By Lemma 5.46 and Lemma 6.2(ii).
\end{proof}

Consider the Concrete Billiard Array $\mathcal{B}$ on $V$ from Lemma 6.3. For location $\lambda=(r,s,t)$ in $\bigtriangleup_{d}$, by Proposition 5.41 the vector $\mathcal{B}_{\lambda}$ is a linear combination of the vectors $u_s, u_{s+1},\dots ,u_{d-r}$ and also a linear combination of the vectors $v_t, v_{t+1},\dots ,v_{d-r}$. For notational convenience, abbreviate $b_i(\lambda)=b_i(\mathcal{B}_{\lambda})$ in (4.2) and $c_i(\lambda)=c_i(\mathcal{B}_{\lambda})$ in (4.3), so that
\begin{equation}
\mathcal{B}_{\lambda}=b_s(\lambda)u_s+b_{s+1}(\lambda)u_{s+1}+\dots + b_{d-r}(\lambda)u_{d-r};
\end{equation}
\begin{equation}
\mathcal{B}_{\lambda}=c_t(\lambda)v_t+c_{t+1}(\lambda)v_{t+1}+\dots + c_{d-r}(\lambda)v_{d-r}.
\end{equation}
In the following result we compute the coefficients in (6.2) in terms of the entries of $T$. The coefficients in (6.1) can be similarly computed, but we don't need these coefficients. Recall the $T[i,j]$ notation from Definition 4.1.
\begin{pro}
With the above notation, for location $\lambda=(r,s,t)$ in $\bigtriangleup_{d}$ we have $c_t(\lambda)=1$. Moreover, if $s> 0$,
\begin{equation}
u=-(T[t+1,d-r])^{-1}v,
\end{equation}
where $u=(c_{t+1}(\lambda),c_{t+2}(\lambda)\dots,c_{d-r}(\lambda))^t$ and $v=(T_{0t},T_{1t}\dots, T_{s-1,t})^t$.
\end{pro}
\begin{proof}
By Lemma 5.47 and Lemma 6.3, $c_t(\lambda)=1$.

For the rest of the proof, assume that $s> 0$.
By (6.1), $b_i(\lambda)=0$ for $0\le i \le s-1$. By (6.2), $c_i(\lambda)=0$ for $0\le i \le t-1$ and $d-r+1 \le i \le d$. Evaluating (4.4) using the above comments, we obtain
\begin{equation}
v+T[t+1,d-r]u=0.
\end{equation}
The matrix $T$ is very good by Definition 4.14, so $T[t+1,d-r]$ is invertible. Solving (6.4) for $u$, we obtain (6.3).
\end{proof}
Recall the $B$-value concept from Definition 5.31. Our next goal is to compute these values for the Billiard Array $B$.
\begin{1def}
Pick $\lambda=(r,s,t)\in \mathbb{Z}^3$. If $\lambda\in \bigtriangleup_{d}$, then define $T[\lambda]$ to be the submatrix $T[t,d-r]$ from Definition 4.1. If $\lambda\notin \bigtriangleup_{d}$, then define $T[\lambda]$ to be the empty set $\emptyset$.
\end{1def}
For the rest of this section assume $d\ge 2$. For a location $\tau=(r,s,t)\in \bigtriangleup_{d-2}$, consider the corresponding white 3-clique in $\bigtriangleup_{d}$ from Lemma 5.12. This white 3-clique consists of the locations
\begin{equation*}
\lambda=(r,s+1,t+1),\qquad \mu=(r+1,s,t+1),\qquad \nu=(r+1,s+1,t).
\end{equation*}
Next, consider the vectors $\mu\pm\alpha, \mu\pm\beta, \mu\pm\gamma,$ where $\alpha, \beta,  \gamma$ are from (5.1). Note that $\lambda=\mu-\alpha$ and $\nu=\mu+\beta$. Moreover,
\begin{equation*}
\begin{split}
\mu+\alpha=(r+2,s-1,t+1), \qquad \mu-\alpha=(r,s+1,t+1),\\
\mu+\beta=(r+1,s+1,t),\qquad \mu-\beta=(r+1,s-1,t+2),\\
\mu+\gamma=(r,s,t+2), \qquad \mu-\gamma=(r+2,s,t).
\end{split}
\end{equation*}
The above vectors form a hexagon as follows:
\[
    \setlength{\arraycolsep}{1pt}
    \begin{array}{ccccc}
        & \mu+\beta  &   & \mu-\alpha &    \\
      \mu-\gamma  &   & \mu &   & \mu+\gamma  \\
        & \mu+\alpha &   & \mu-\beta &      \\
    \end{array}
  \]

\begin{thm}
With the above notation, the $B$-value of the white 3-clique in $\bigtriangleup_{d}$ that corresponds to $\tau$ is
\begin{equation}
\frac{\det(T[\mu+\alpha])\det(T[\mu+\beta])\det(T[\mu+\gamma])}{\det(T[\mu-\alpha])\det(T[\mu-\beta])\det(T[\mu-\gamma])},
\end{equation}
where we interpret $\det(\emptyset)=1$.
\end{thm}
\begin{proof}
Consider the following vectors:
\[
    \setlength{\arraycolsep}{1pt}
    \begin{array}{ccccc}
     &  & \mu-\alpha+\beta  &  &      \\
        & \mu+\beta  &   & \mu-\alpha &    \\
      \mu-\gamma  &   & \mu &   & \mu+\gamma  \\
    \end{array}
  \]
We have three black 3-cliques with the following locations:
\begin{enumerate}
\item[\rm(i)] $\mu,\ \mu-\alpha$, $\mu+\gamma$;
\item[\rm(ii)] $\mu-\gamma,\ \mu+\beta$, $\mu$;
\item[\rm(iii)] $\mu+\beta,\ \mu-\alpha+\beta$, $\mu-\alpha$.
\end{enumerate}
For notational convenience, let $\overline{\tau}$ denote the $B$-value of the white 3-clique in $\bigtriangleup_{d}$ that corresponds to $\tau$.
By Lemma 5.40,
\begin{equation}
\overline{\tau}=\widetilde{\mathcal{B}}_{\mu-\alpha,\mu}\widetilde{\mathcal{B}}_{\mu,\mu+\beta}\widetilde{\mathcal{B}}_{\mu+\beta,\mu-\alpha}.
\end{equation}
We now show that
\begin{equation}
\widetilde{\mathcal{B}}_{\mu-\alpha,\mu}=-1.
\end{equation}
We apply Lemma 5.39 to the black 3-clique (i) and obtain
\begin{equation}
\mathcal{B}_{\mu-\alpha}+\widetilde{\mathcal{B}}_{\mu-\alpha,\mu}\mathcal{B}_{\mu}+\widetilde{\mathcal{B}}_{\mu-\alpha,\mu+\gamma}\mathcal{B}_{\mu+\gamma}=0.
\end{equation}
For convenience we rewrite (6.8) as
\begin{equation}
\mathcal{B}_{\mu-\alpha}-\mathcal{B}_{\mu}+(\widetilde{\mathcal{B}}_{\mu-\alpha,\mu}+1)\mathcal{B}_{\mu}+\widetilde{\mathcal{B}}_{\mu-\alpha,\mu+\gamma}\mathcal{B}_{\mu+\gamma}=0.
\end{equation}
By Proposition 5.41,
\begin{equation}
\mathcal{B}_{\mu}\in \mathbb{F}v_{t+1}+\mathbb{F} v_{t+2}+\dots +\mathbb{F}v_{d-r-1},
\end{equation}
\begin{equation}
\mathcal{B}_{\mu-\alpha}\in \mathbb{F}v_{t+1}+\mathbb{F} v_{t+2}+\dots +\mathbb{F}v_{d-r},
\end{equation}
\begin{equation}
\mathcal{B}_{\mu+\gamma}\in \mathbb{F}v_{t+2}+\mathbb{F} v_{t+3}+\dots +\mathbb{F}v_{d-r}.
\end{equation}
By Proposition 6.4,
\begin{equation}
c_{t+1}(\mu)=1, \qquad c_{t+1}(\mu-\alpha)=1.
\end{equation}
By (6.10), (6.11) and (6.13),
\begin{equation}
\mathcal{B}_{\mu-\alpha}-\mathcal{B}_{\mu}\in \mathbb{F}v_{t+2}+\mathbb{F} v_{t+3}+\dots +\mathbb{F}v_{d-r}.
\end{equation}
By (6.9), (6.12) and (6.14),
\begin{equation}
(\widetilde{\mathcal{B}}_{\mu-\alpha,\mu}+1)\mathcal{B}_{\mu}\in \mathbb{F}v_{t+2}+\mathbb{F} v_{t+3}+\dots +\mathbb{F}v_{d-r}.
\end{equation}
By (6.10), the equation on the left in (6.13), and (6.15), we obtain (6.7).\\
Next, we show that
\begin{equation}
\widetilde{\mathcal{B}}_{\mu,\mu+\beta}=\frac{\det(T[\mu+\alpha])\det(T[\mu])}{\det(T[\mu-\beta])\det(T[\mu-\gamma])}.
\end{equation}
There are two cases. First assume that $s\ne 0$. We apply Lemma 5.39 to the black 3-clique (ii) and obtain
\begin{equation}
\mathcal{B}_{\mu}+\widetilde{\mathcal{B}}_{\mu,\mu+\beta}\mathcal{B}_{\mu+\beta}+\widetilde{\mathcal{B}}_{\mu,\mu-\gamma}\mathcal{B}_{\mu-\gamma}=0.
\end{equation}
By Proposition 5.41,
\begin{equation}
\mathcal{B}_{\mu-\gamma}\in \mathbb{F}v_{t}+\mathbb{F} v_{t+1}+\dots +\mathbb{F}v_{d-r-2}.
\end{equation}
By (6.17) and (6.18),
\begin{equation}
\mathcal{B}_{\mu}+\widetilde{\mathcal{B}}_{\mu,\mu+\beta}\mathcal{B}_{\mu+\beta}\in \mathbb{F}v_{t}+\mathbb{F} v_{t+1}+\dots +\mathbb{F}v_{d-r-2}.
\end{equation}
By Proposition 5.41,
\begin{equation}
\begin{split}
\mathcal{B}_{\mu}\in \mathbb{F}v_{t+1}+\mathbb{F} v_{t+2}+\dots +\mathbb{F}v_{d-r-1};\\
\mathcal{B}_{\mu+\beta}\in \mathbb{F}v_{t}+\mathbb{F} v_{t+1}+\dots +\mathbb{F}v_{d-r-1}.\\
\end{split}
\end{equation}
By Proposition 6.4 and Cramer's rule,
\begin{equation}
\begin{split}
c_{d-r-1}(\mu)=\frac{(-1)^{s}\det(T[t+1,d-r-2])}{\det(T[t+2,d-r-1])};\\
c_{d-r-1}(\mu+\beta)=\frac{(-1)^{s+1}\det(T[t,d-r-2])}{\det(T[t+1,d-r-1])}.\\
\end{split}
\end{equation}
By (6.19)--(6.21),
\begin{equation}
\widetilde{\mathcal{B}}_{\mu,\mu+\beta} = \frac{\det(T[t+1,d-r-2])\det(T[t+1,d-r-1])}{\det(T[t+2,d-r-1])\det(T[t,d-r-2])}.
\end{equation}
Evaluating (6.22) using Definition 6.5, we obtain (6.16).\\
Next assume that $s=0$. Using an argument similar to (6.17)--(6.21), we obtain
\begin{equation}
\widetilde{\mathcal{B}}_{\mu,\mu+\beta}=\frac{\det(T[t+1,t+1])}{\det(T[t,t])}.
\end{equation}
Since we interpret $\det(\emptyset)=1$,
\begin{equation}
\det(T[\mu+\alpha])=1, \qquad \det(T[\mu-\beta])=1.
\end{equation}
Evaluate (6.23) using Definition 6.5. Combining the result with (6.24), we obtain (6.16).\\
We have shown (6.16). In a similar manner using the black 3-clique (iii), we obtain
\begin{equation}
\widetilde{\mathcal{B}}_{\mu+\beta,\mu-\alpha}=-\frac{\det(T[\mu+\beta])\det(T[\mu+\gamma])}{\det(T[\mu-\alpha])\det(T[\mu])}.
\end{equation}
Evaluating (6.6) using (6.7), (6.16), (6.25) we obtain (6.5).
\end{proof}

\section{How $\mathcal{T}_d(\mathbb{F})$ is related to $BA_d{(\mathbb{F}})$}
Recall the set $\mathcal{T}_d(\mathbb{F})$ from Definition 4.14, and the set $BA_d{(\mathbb{F}})$ from Definition 5.34. In this section we explain how $\mathcal{T}_d(\mathbb{F})$ is related to $BA_d{(\mathbb{F}})$. For convenience, we first consider an equivalence relation on $\mathcal{T}_d(\mathbb{F})$.
\begin{1def}
For $T, T^\prime\in \mathcal{T}_d(\mathbb{F})$, we declare $T\sim T^\prime$ whenever there exist invertible diagonal matrices $H, K \in {\rm Mat}_{d+1}(\mathbb{F})$ such that
\begin{equation}
T^{\prime}=HTK.
\end{equation}
The relation $\sim$ is an equivalence relation. For $T\in \mathcal{T}_d(\mathbb{F})$, let $[T]$ denote the equivalence class of $\sim$ that contains $T$. Let $\mathbb{T}_d(\mathbb{F})$ denote the set of equivalence classes for $\sim$.
\end{1def}
Recall the map $b: \mathcal{T}_d(\mathbb{F}) \to BA_d(\mathbb{F})$ from Definition 6.1. As we will see, $b$ is surjective but not bijective. We will show that for $T, T^\prime\in \mathcal{T}_d(\mathbb{F})$, $b(T)=b(T^\prime)$ if and only if $T \sim T^\prime$. This tells us that the $b$-induced map $\mathbb{T}_d(\mathbb{F})\to BA_d(\mathbb{F})$ is bijective.

Referring to Definition 7.1, assume that $T \sim T^\prime$. Pick invertible diagonal matrices $H, K \in {\rm Mat}_{d+1}(\mathbb{F})$ that satisfy (7.1). Observe that the entries $H_{ii}\ne 0$ and $K_{ii}\ne 0$ for $0\le i \le d$. View $T$ as the the transition matrix from a basis $\{u_i\}_{i=0}^{d}$ of $V$ to a basis $\{v_i\}_{i=0}^{d}$ of $V$ as around (4.1). Recall the flags $\{U_{i}\}_{i=0}^{d}$, $\{U_{i}^{\prime}\}_{i=0}^{d}$, $\{U_{i}^{\prime\prime}\}_{i=0}^{d}$ from Definition 6.1. The matrix $T^\prime$ can be viewed as the transition matrix from a basis $\{u_i^\prime\}_{i=0}^{d}$ of $V$ to a basis $\{v_i^\prime\}_{i=0}^{d}$ of $V$, where $u_i=H_{ii}u_i^\prime$ and $K_{ii}v_i=v_i^\prime$ for $0\le i\le d$. Observe that the flags $\{U_{i}\}_{i=0}^{d}$, $\{U_{i}^{\prime}\}_{i=0}^{d}$, $\{U_{i}^{\prime\prime}\}_{i=0}^{d}$ are the same for $T$ and $T^\prime$.
\begin{lem}
Let matrices $T,T^{\prime}\in \mathcal{T}_d(\mathbb{F})$ satisfy $T\sim T^\prime$ in the sense of Definition 7.1. Then $b(T)=b(T^\prime)$.
\end{lem}
\begin{proof}
By the discussion above the lemma statement, along with Definition 6.1.
\end{proof}
\begin{1def}
Using the map $b: \mathcal{T}_d(\mathbb{F}) \to BA_d(\mathbb{F})$, we define a map ${\bf b}:  \mathbb{T}_d(\mathbb{F}) \to BA_d(\mathbb{F})$ as follows. Given an equivalence class $[T] \in \mathbb{T}_d(\mathbb{F})$, the image of $[T]$ under ${\bf b}$ is $b(T)$. By Lemma 7.2 the map ${\bf b}$ is well-defined.
\end{1def}
\begin{thm}
The map ${\bf b}: \mathbb{T}_d(\mathbb{F}) \to BA_d(\mathbb{F})$ from Definition 7.3 is bijective.
\end{thm}
\begin{proof}
First we show that the map ${\bf b}$ is surjective. Pick a Billiard Array $B$ on $V$. It suffices to show that there exists $T\in \mathcal{T}_d(\mathbb{F})$ such that $B\in b(T)$. For $0\le i \le d$, pick $0\ne u_i\in B_\mu$ where $\mu=(d-i,i,0)\in \bigtriangleup_{d}$, and $0\ne v_i\in B_\nu$ where $\nu=(d-i,0,i)\in \bigtriangleup_{d}$. By Definition 5.14, $\{u_i\}_{i=0}^d$ and $\{v_i\}_{i=0}^d$ are bases of $V$. Let $T\in {\rm Mat}_{d+1}(\mathbb{F})$ denote the transition matrix from the basis $\{u_i\}_{i=0}^d$ to the basis $\{v_i\}_{i=0}^d$. By Corollary 5.44 $T\in \mathcal{T}_d(\mathbb{F})$. Above Lemma 6.2 we refer to a Billiard Array obtained from Theorem 5.22. By Lemma 6.2, this Billiard Array is the Billiard Array $B$. Hence $B\in b(T)$ by the last sentence above Lemma 6.2. We have shown that the map ${\bf b}$ is surjective.

Next we show that the map ${\bf b}$ is injective. Suppose that $T,T^\prime \in \mathcal{T}_d(\mathbb{F})$ satisfy $b(T)=b(T^\prime)$. We will show that $T\sim T^\prime$ in the sense of Definition 7.1. Define the Billiard Arrays $B\in b(T)$ and $B^\prime \in b(T^\prime)$ as around Definition 6.1. Since $b(T)=b(T^\prime)$, the Billiard Arrays $B$ and $B^\prime$ are isomorphic. Therefore there exists an isomorphism $\sigma$ of Billiard Arrays from $B$ to $B^\prime$. By Definition 5.16, $\sigma: V \to V$ is an $\mathbb{F}$-linear bijection that sends $B_\lambda\mapsto B_\lambda^\prime$ for all $\lambda\in \bigtriangleup_{d}$. Associated with $T$ we have the bases $\{u_i\}_{i=0}^{d}$ and $\{v_i\}_{i=0}^{d}$ of $V$ from Definition 6.1. Similarly, associated with $T^\prime$ we have bases $\{u_i^\prime\}_{i=0}^{d}$ and $\{v_i^\prime\}_{i=0}^{d}$ of $V$. To be more precise, $T$ is the transition matrix from the basis $\{u_i\}_{i=0}^d$ to the basis $\{v_i\}_{i=0}^d$, and $T^\prime$ is the transition matrix from the basis $\{u_i^\prime\}_{i=0}^d$ to the basis $\{v_i^\prime\}_{i=0}^d$. Since $\sigma$ is an $\mathbb{F}$-linear bijection, $T$ is also the transition matrix from the basis $\{\sigma(u_i)\}_{i=0}^d$ of $V$ to the basis $\{\sigma(v_i)\}_{i=0}^d$ of $V$. By Lemma 6.2, for $0\le i\le d$ there exist nonzero $h_i,k_i\in \mathbb{F}$ such that $\sigma(u_i)=h_iu_i^\prime$ and $k_i\sigma(v_i)=v_i^\prime$. Define diagonal matrices $H, K \in {\rm Mat}_{d+1}(\mathbb{F})$ such that $H_{ii}=h_i$ and $K_{ii}=k_i$ for $0\le i\le d$. By construction $H$ and $K$ are invertible. By the above comments, $T^\prime=HTK$. By Definition 7.1, $T\sim T^\prime$. We have shown that the map ${\bf b}$ is injective. Hence the map ${\bf b}$ is bijective.
\end{proof}
We continue to discuss the equivalence relation $\sim$ from Definition 7.1. Our next goal is to identify a representative in each equivalence class.
\begin{1def}
A matrix $T\in \mathcal{T}_d(\mathbb{F})$ is called {\it nice} whenever $T_{0i}=T_{ii}=1$ for $0\le i\le d$.
\end{1def}
\begin{lem}
For the equivalence relation $\sim$ on $\mathcal{T}_d(\mathbb{F})$ from Definition 7.1, each equivalence class contains a unique nice element in the sense of Definition 7.5.
\end{lem}
\begin{proof}
First, we show that each equivalence class in $\mathcal{T}_d(\mathbb{F)}$ contains at least one nice element.
Choose $T\in \mathcal{T}_d(\mathbb{F})$. By Definition 4.14, we have $T_{0i}\ne 0$ and $T_{ii}\ne 0$ for $0\le i\le d$. Define diagonal matrices $H, K \in {\rm Mat}_{d+1}(\mathbb{F})$ such that $H_{ii}=T_{0i}/T_{ii}$ and $K_{ii}=1/T_{0i}$ for $0\le i\le d$. By construction $H$ and $K$ are invertible. Let $T^{\prime}=HTK$. By construction, $T^\prime\in \mathcal{T}_d(\mathbb{F})$ and $T^\prime_{0i}=T^\prime_{ii}=1$ for $0\le i\le d$. By Definition 7.1, we have $T \sim T^\prime$. By Definition 7.5, $T^\prime$ is nice. Therefore each equivalence class in $\mathcal{T}_d(\mathbb{F)}$ contains at least one nice element.

Next, we show that this nice element is unique. Suppose that $T$, $T^\prime$ are nice elements in $\mathcal{T}_d(\mathbb{F)}$ and $T\sim T^\prime$. By Definition 7.1, there exist invertible diagonal matrices $H, K \in {\rm Mat}_{d+1}(\mathbb{F})$ that satisfy (7.1). For $0\le i\le d$, examining the $(0,i)$-entry in (7.1), we obtain $T^\prime_{0i}=H_{00}T_{0i}K_{ii}$. Examining the $(i,i)$-entry in (7.1), we obtain $T^\prime_{ii}=H_{ii}T_{ii}K_{ii}$. By Definition 7.5, we have $T_{0i}=T_{ii}=T^\prime_{0i}=T^\prime_{ii}=1$. By the above comments, we have $K_{ii}=1/H_{00}$ and $H_{ii}=H_{00}$ for $0\le i\le d$. Therefore $H=H_{00}I$ and $K=I/H_{00}$. Consequently $T=T^\prime$ by (7.1). We have shown that each equivalence class in $\mathcal{T}_d(\mathbb{F)}$ contains a unique nice element.
\end{proof}

\section{A commutative diagram}
Recall the set $\mathcal{T}_d(\mathbb{F})$ from Definition 4.14, the set $BA_d{(\mathbb{F}})$ from Definition 5.34, and the set $VF_{d}(\mathbb{F})$ from Definition 5.35. For the moment assume $d\ge 2$. In this section, we will describe how the sets $\mathcal{T}_d(\mathbb{F})$, $BA_d{(\mathbb{F}})$, $VF_{d}(\mathbb{F})$ and $VF_{d-2}(\mathbb{F})$ are related. In order to do this, we will establish a commutative diagram. As we proceed, some of our results do not require $d\ge 2$. So until further notice, assume $d\ge 0$. First we define a map $D: \mathcal{T}_d(\mathbb{F}) \to VF_{d}(\mathbb{F})$.
\begin{1def}
For $T\in \mathcal{T}_d(\mathbb{F})$, define a function $D(T): \bigtriangleup_{d} \to \mathbb{F}$ as follows. For each location $\lambda=(r,s,t)\in \bigtriangleup_{d}$, the image of $\lambda$ under $D(T)$ is $\det(T[\lambda])$, where $T[\lambda]$ is from Definition 6.5.
\end{1def}

\begin{lem}
With reference to Definition 8.1, $D(T)$ is a value function on $\bigtriangleup_{d}$ in the sense of Definition 5.32. In other words, $D(T)\in VF_{d}(\mathbb{F})$.
\end{lem}
\begin{proof}
Pick $\lambda=(r,s,t)\in\bigtriangleup_{d}$. By construction $D(T)(\lambda)=\det(T[t,d-r])$, which is nonzero since $T$ is very good. We have shown that $D(T)(\lambda)\ne 0$ for all $\lambda\in\bigtriangleup_{d}$. Therefore $D(T)\in VF_{d}(\mathbb{F})$.
\end{proof}
\begin{1def}
We define a map $D: \mathcal{T}_d(\mathbb{F}) \to VF_{d}(\mathbb{F})$ as follows. For $T\in \mathcal{T}_d(\mathbb{F})$, the image of $T$ under $D$ is the function $D(T)$ from Definition 8.1. By Lemma 8.2 the map $D$ is well defined.
\end{1def}
For later use, we recall an elementary fact from linear algebra. Pick $T\in {\rm Mat}_{d+1}(\mathbb{F})$. For $0\le i,j\le d$, let $T^{(i,j)}$ denote the determinant of the $d \times d$ matrix that results from deleting the $i$-th row and the $j$-th column of $T$. Applying the Laplace expansion to the bottom row of $T$ we obtain
\begin{equation}
\det(T)=\sum_{j=0}^d (-1)^{d+j}T_{dj}T^{(d,j)}.
\end{equation}
\begin{lem}
The map $D: \mathcal{T}_d(\mathbb{F}) \to VF_{d}(\mathbb{F})$ from Definition 8.3 is bijective.
\end{lem}
\begin{proof}
For $f\in VF_{d}(\mathbb{F})$, we show that there exists a unique $T \in \mathcal{T}_d(\mathbb{F})$ such that $D(T)=f$. Our strategy is as follows. We are going to show that $T$ exists. As we proceed, our calculation will show that there is only one solution for $T$. For $0\le i,j\le d$, we solve for $T_{ij}$ by induction on $i+j$. First assume that $i+j=0$, so that $i=j=0$. Then $D(T)=f$ forces $T_{00}=f((d,0,0))$. Next assume that $i+j> 0$. If $i>j$, then $T_{ij}=0$ since $T$ is required to be upper triangular. If $i=0$, then $D(T)=f$ forces $T_{0j}=f((d-j,0,j))$. If $0<i\le j$, then $D(T)=f$ forces $\det(T[j-i,j])=f((d-j,i,j-i))$. Consider the matrix $T[j-i,j]$. The bottom right entry is $T_{ij}$. All the other entries have already been computed by induction. We now compute $\det(T[j-i,j])$ using the Laplace expansion to its bottom row. Applying (8.1) to $T[j-i,j]$, we obtain a formula for $\det(T[j-i,j])$. In this formula the coefficient of $T_{ij}$ is $\det(T[j-i,j-1])$, which is nonzero by assumption. Therefore there is a unique solution for $T_{ij}$. We have shown that the map $D$ is bijective.
\end{proof}
From now until the end of Theorem 8.6, assume that $d\ge 2$. Next we define a map $w: VF_{d}(\mathbb{F}) \to VF_{d-2}(\mathbb{F})$.

\begin{1def}
Assume $d\ge 2$. We define a map $w: VF_{d}(\mathbb{F}) \to VF_{d-2}(\mathbb{F})$ as follows. Given $f\in VF_{d}(\mathbb{F})$, we describe the image of $f$ under $w$. For a location $\tau=(r,s,t)\in \bigtriangleup_{d-2}$, the value of $w(f)(\tau)$ is
\begin{equation}
\frac{f(\mu+\alpha)f(\mu+\beta)f(\mu+\gamma)}{f(\mu-\alpha)f(\mu-\beta)f(\mu-\gamma)}.
\end{equation}
Here $\mu=(r+1,s,t+1)\in \bigtriangleup_{d}$ and $\alpha,\beta,\gamma$ are from (5.1). We interpret $f(\lambda)=1$ for all $\lambda\in \mathbb{Z}^3$ such that $\lambda \not\in \bigtriangleup_{d}$.
\end{1def}

\begin{thm}
Assume $d\ge 2$. Then the following diagram commutes:
\[
\xymatrix{
\mathcal{T}_d(\mathbb{F}) \ar[d]_{D} \ar[r]^b &BA_d(\mathbb{F})\ar[d]^\theta\\
VF_{d}(\mathbb{F}) \ar[r]_w & VF_{d-2}(\mathbb{F})}
\]
Here $\theta$ is from Definition 5.36, $b$ is from Definition 6.1, $D$ is from Definition 8.3, and $w$ is from Definition 8.5.
\end{thm}
\begin{proof}
For $T\in \mathcal{T}_d(\mathbb{F})$ chase $T$ around the diagram using Theorem 6.6.
\end{proof}
Recall the equivalence relation $\sim$ on $\mathcal{T}_d(\mathbb{F})$ from Definition 7.1.
\begin{1def}
Via the bijection $D: \mathcal{T}_d(\mathbb{F}) \to VF_{d}(\mathbb{F})$, the equivalence relation $\sim$ on $\mathcal{T}_d(\mathbb{F})$ induces an equivalence relation on $VF_{d}(\mathbb{F})$, which we denote by $\approx$. In other words, for $T, T^\prime\in \mathcal{T}_d(\mathbb{F})$, $T\sim T^\prime$ if and only if $D(T)\approx D(T^\prime)$. For $f\in VF_d(\mathbb{F})$, let $[f]$ denote the equivalence class of $\approx$ that contains $f$. Let $\mathbb{VF}_d(\mathbb{F})$ denote the set of equivalence classes for $\approx$.
\end{1def}
We now give an alternative description of $\approx$.
\begin{lem}
For $f,f^\prime \in VF_{d}(\mathbb{F})$ the following are equivalent:
\begin{enumerate}
\item[\rm(i)] $f\approx f^\prime$ in the sense of Definition 8.7;
\item[\rm(ii)] there exist nonzero $h_i,k_i\in \mathbb{F}$ $(0\le i \le d)$ such that
\begin{equation}
f^\prime(\lambda)=f(\lambda)\prod_{i=0}^s(h_ik_{t+i})
\end{equation}
for $\lambda=(r,s,t)\in \bigtriangleup_{d}$.
\end{enumerate}
\end{lem}
\begin{proof}
By Lemma 8.4 there exist $T,T^\prime \in \mathcal{T}_d(\mathbb{F})$ such that  $D(T)=f$ and $D(T^\prime)=f^\prime$.

\rm(i)$\Rightarrow$\rm(ii) By Definition 8.7 we have $T\sim T^\prime$. So there exist invertible diagonal matrices $H, K \in {\rm Mat}_{d+1}(\mathbb{F})$ that satisfy (7.1). Define $h_i=H_{ii}$ and $k_i=K_{ii}$ for $0\le i \le d$. By construction $0\ne h_i,k_i\in \mathbb{F}$ for $0\le i\le d$, and $T^\prime_{ij}=h_ik_jT_{ij}$ for $0\le i,j\le d$. Pick $\lambda=(r,s,t)\in \bigtriangleup_{d}$. By Definition 8.1 and the above comments,
\begin{eqnarray*}
f^\prime(\lambda)&=&\det(T^\prime[t,d-r])\\
&=&h_0h_1\dots h_sk_tk_{t+1}\dots k_{s+t}\det(T[t,d-r])\\
&=&h_0h_1\dots h_sk_tk_{t+1}\dots k_{s+t}f(\lambda).
\end{eqnarray*}
Therefore (8.3) holds.

\rm(ii)$\Rightarrow$\rm(i) Define diagonal matrices $H, K \in {\rm Mat}_{d+1}(\mathbb{F})$ such that $H_{ii}=h_i$ and $K_{ii}=k_i$ for $0\le i\le d$. By construction $H,K$ are invertible. We will show that $T^\prime=HTK$. In order to do this, we show that $T_{ij}^\prime=h_ik_jT_{ij}$ for $0\le i,j\le d$. We proceed by induction on $i+j$. First assume that $i+j=0$, so that $i=j=0$. Applying (8.3) to $\lambda=(d,0,0)\in \bigtriangleup_{d}$ we obtain $T_{00}^\prime=h_0k_0T_{00}$. Next assume that $i+j>0$. If $i>j$, then $T_{ij}^\prime=T_{ij}=0$. Therefore $T_{ij}^\prime=h_ik_jT_{ij}$. If $i=0$, then by applying (8.3) to $\lambda=(d-j,0,j)\in \bigtriangleup_{d}$ we obtain $T_{0j}^\prime=h_0k_jT_{0j}$. If $0<i\le j$, then by applying (8.3) to $\lambda=(d-j,i,j-i)\in \bigtriangleup_{d}$ we obtain
\begin{equation}
\det(T^\prime[j-i,j])=\det(T[j-i,j])\prod_{l=0}^i(h_lk_{j-i+l}).
\end{equation}
By (8.4) and induction we routinely obtain $T_{ij}^\prime=h_ik_jT_{ij}$. We have shown that $T^\prime=HTK$. By Definition 7.1, $T\sim T^\prime$. By Definition 8.7, $f\approx f^\prime$.
\end{proof}
\begin{1def}
An element $f\in VF_{d}(\mathbb{F})$ is called {\it fine} whenever $f((i,d-i,0))=f((i,0,d-i))=1$ for $0\le i\le d$.
\end{1def}

\begin{lem}
For $T\in \mathcal{T}_d(\mathbb{F})$ the following are equivalent:
\begin{enumerate}
\item[\rm(i)] $T$ is nice in the sense of Definition 7.5;
\item[\rm(ii)] $D(T)$ is fine in the sense of Definition 8.9.
\end{enumerate}
\end{lem}
\begin{proof}
By Definition 8.1,
\begin{equation}
D(T)((i,d-i,0))=T_{00}T_{11}\dots T_{d-i,d-i}. \qquad (0\le i\le d)
\end{equation}
Similarly, by Definition 8.1,
\begin{equation}
D(T)((i,0,d-i))=T_{0,d-i}. \qquad (0\le i\le d)
\end{equation}

First assume that $T$ is nice. By Definition 7.5 along with (8.5) and (8.6), $D(T)((i,d-i,0))=D(T)((i,0,d-i))=1$ for $0\le i\le d$. Therefore $D(T)$ is fine in view of Definition 8.9.

Next assume that $D(T)$ is fine. By Definition 8.9 along with (8.5) and (8.6), $T_{0i}=T_{ii}=1$ for $0\le i\le d$. Therefore $T$ is nice in view of Definition 7.5.
\end{proof}
\begin{cor}
Under the equivalence relation $\approx$ from Definition 8.7, each equivalence class contains a unique fine element in $VF_{d}(\mathbb{F})$.
\end{cor}
\begin{proof}
By Lemma 7.6 and Lemma 8.10.
\end{proof}

\begin{1def}
We define a map ${\bf D}:  \mathbb{T}_d(\mathbb{F}) \to \mathbb{VF}_d(\mathbb{F})$ as follows. Given an equivalence class $[T]\in \mathbb{T}_d(\mathbb{F})$, the image of $[T]$ under ${\bf D}$ is $[D(T)]$. By Definition 8.7 the map ${\bf D}$ is well-defined.
\end{1def}
\begin{lem}
The map ${\bf D}:  \mathbb{T}_d(\mathbb{F}) \to \mathbb{VF}_d(\mathbb{F})$ from Definition 8.12 is bijective.
\end{lem}
\begin{proof}
By Lemma 8.4 and Definition 8.7.
\end{proof}
For the rest of this section assume $d\ge 2$.
\begin{lem}
Assume $d\ge 2$. Recall the map $w$ from Definition 8.5. For $f,f^\prime \in VF_{d}(\mathbb{F})$, suppose that $f\approx f^\prime$ in the sense of Definition 8.7. Then $w(f)=w(f^\prime)$.
\end{lem}
\begin{proof}
Apply (8.2) to $f^\prime$ and evaluate the result using (8.3).
\end{proof}
\begin{1def}
Assume $d\ge 2$. We define a map ${\bf w}:  \mathbb{VF}_d(\mathbb{F}) \to VF_{d-2}(\mathbb{F})$ as follows. Given an equivalence class $[f]\in \mathbb{VF}_d(\mathbb{F})$, the image of $[f]$ under ${\bf w}$ is $w(f)$. By Lemma 8.14 the map ${\bf w}$ is well-defined.
\end{1def}

\begin{lem}
Assume $d\ge 2$. Then the map ${\bf w}:  \mathbb{VF}_d(\mathbb{F}) \to VF_{d-2}(\mathbb{F})$ from Definition 8.15 is bijective.
\end{lem}
\begin{proof}
We will show that for $g\in VF_{d-2}(\mathbb{F})$, there exists a unique fine $f\in VF_{d}(\mathbb{F})$ such that $w(f)=g$. Our strategy is as follows. We are going to show that $f$ exists. As we proceed, our calculation will show that there is only one solution for $f$. For $\lambda=(r,s,t)\in \bigtriangleup_{d}$, we solve for $f(\lambda)$ by induction on $s-r$. First assume that $s-r=-d$, so that $\lambda=(d,0,0)$. By Definition 8.9, $f(\lambda)=1$. Next assume that $s-r>-d$. If $s=0$ or $t=0$, then by Definition 8.9, $f(\lambda)=1$. If $s\ne 0$ and $t\ne 0$, then by Definition 8.5,
\begin{equation}
g(\tau)=\frac{f(\mu+\alpha)f(\mu+\beta)f(\mu+\gamma)}{f(\mu-\alpha)f(\mu-\beta)f(\mu-\gamma)},
\end{equation}
where $\tau=(r,s-1,t-1)\in \bigtriangleup_{d-2}$, $\mu=(r+1,s-1,t)\in \bigtriangleup_{d}$ and $\alpha,\beta,\gamma$ are from (5.1). In terms of $r,s,t$ the equation (8.7) becomes
\begin{equation}
g(r,s-1,t-1)=\frac{f(r+2,s-2,t)f(r+1,s,t-1)f(r,s-1,t+1)}{f(r,s,t)f(r+1,s-2,t+1)f(r+2,s-1,t-1)}.
\end{equation}
In the right hand side of (8.8), all the factors except for $f(\lambda)=f(r,s,t)$ have been determined by induction, and these factors are nonzero. Hence $f(\lambda)=f(r,s,t)$ is uniquely determined. We have shown that for $g\in VF_{d-2}(\mathbb{F})$, there exists a unique fine $f\in VF_{d}(\mathbb{F})$ such that $w(f)=g$. By this and Corollary 8.11, the map ${\bf w}$ is bijective.
\end{proof}
\begin{thm}
Assume $d\ge 2$. Then the following diagram commutes:
\[
\xymatrix{
\mathbb{T}_d(\mathbb{F}) \ar[d]_{{\bf D}} \ar[r]^{{\bf b}} &BA_d(\mathbb{F})\ar[d]^\theta\\
\mathbb{VF}_d(\mathbb{F}) \ar[r]_{{\bf w}} & VF_{d-2}(\mathbb{F})}
\]
Here $\theta$ is from Definition 5.36, ${\bf b}$ is from Definition 7.3, ${\bf D}$ is from Definition 8.12, and ${\bf w}$ is from Definition 8.15.
\end{thm}
\begin{proof}
Use Theorem 8.6.
\end{proof}
\begin{rmk}
In the commutative diagram from Theorem 8.17, each map is bijective.
\end{rmk}

\section{an example}
In this section, we will give an example to illustrate Theorem 6.6. We first recall some notation. Fix $0\ne q\in \mathbb{F}$. For $n\in \mathbb{N}$ define
\begin{equation*}
[n]_q = \sum_{i=0}^{n-1}q^i, \qquad \qquad [n]_q^! = \prod_{i=1}^{n}[i]_q.
\end{equation*}
We interpret $[0]_q=0$ and $[0]_q^!=1$.\\
For $n,k\in \mathbb{Z}$ we define ${n \brack k}_q$ as follows. For $0\le k \le n$,
\begin{equation}
{n \brack k}_q =\frac{[n]_q^!}{[k]_q^![n-k]_q^!}.
\end{equation}
If $k<0$ or $k>n$, then for notational convenience define ${n \brack k}_q=0$.
\begin{rmk}
In the right-hand side of (9.1), for certain values of $q$ the denominator may be equal to $0$. However, by \cite[Theorem~6.1]{kac} ${n \brack k}_q$ is a polynomial in $q$ with integral coefficients. So (9.1) is well defined no matter which $q$ we have chosen.
\end{rmk}
We will use the following result.
\begin{lem}
\cite[Proposition~6.1]{kac} For $n,k\in \mathbb{N}$,
\begin{equation*}
{n+1\brack k}_q={n\brack k}_q+q^{n+1-k}{n\brack k-1}_q.
\end{equation*}
\end{lem}
For the rest of this section assume $d\ge 2$. Recall the set $\mathcal{T}_d(\mathbb{F})$ from Definition 4.14 and the map $b$ from Definition 6.1. We will display a matrix $T\in \mathcal{T}_d(\mathbb{F})$ which is nice in the sense of Definition 7.5. Then we show that for a Billiard Array $B\in b(T)$ the $B$-value of each white 3-clique in $\bigtriangleup_{d}$ is equal to $q^{-1}$.

Define an upper triangular matrix $T\in {\rm Mat}_{d+1}(\mathbb{F})$ as follows. For $0\le i\le j\le d$,
\begin{equation}
T_{ij}={j\brack i}_q.
\end{equation}
\begin{lem}
For $0\le i\le d$, $T_{0i}=T_{ii}=1$.
\end{lem}
\begin{proof}
By (9.2).
\end{proof}

\begin{lem}
For $1\le i\le d$ and $0\le j\le d-1$,
\begin{equation}
T_{i,j+1}-T_{ij}=q^{j-i+1}T_{i-1,j}.
\end{equation}
\end{lem}
\begin{proof}
By (9.2) and Lemma 9.2.
\end{proof}
For later use, we recall an elementary fact from linear algebra. Pick $A\in {\rm Mat}_{d+1}(\mathbb{F})$. For $0\le i\le d$, let $A_i$ denote the $i$-th column of $A$. Let $A^\prime$ denote the matrix in ${\rm Mat}_{d+1}(\mathbb{F})$ with $i$-th column $A_i-A_{i-1}$ for $1\le i\le d$ and $0$-th column $A_0$. Then
\begin{equation}
\det(A)=\det(A^\prime).
\end{equation}
For $0\le i \le j \le d$, recall the submatrix $T[i,j]$ from Definition 4.1. We now compute $\det(T[i,j])$.
\begin{lem}
For $0\le i<j\le d$,
\begin{equation*}
\det(T[i,j])=q^{i(j-i)}\det(T[i,j-1]).
\end{equation*}
\end{lem}
\begin{proof}
Apply (9.4) to $A=T[i,j]$ and use (9.3).
\end{proof}
\begin{pro}
For $0\le i\le j\le d$,
\begin{equation*}
\det(T[i,j])=q^{i(j-i)(j-i+1)/2}.
\end{equation*}
\end{pro}
\begin{proof}
Use Lemma 9.5 and induction on $j-i$.
\end{proof}
We now restate Proposition 9.6 using the notation $T[\lambda]$ from Definition 6.5.
\begin{pro}
For $\lambda=(r,s,t)\in \bigtriangleup_{d}$,
\begin{equation}
\det(T[\lambda])=q^{ts(s+1)/2},
\end{equation}
where $T[\lambda]$ is from Definition 6.5.
\end{pro}
\begin{proof}
Use Definition 6.5 and Proposition 9.6.
\end{proof}

\begin{lem}
$T$ is very good in the sense of Definition 4.4. Moreover, $T\in \mathcal{T}_d(\mathbb{F})$.
\end{lem}
\begin{proof}
By construction $T$ is upper triangular. By Proposition 9.6, $\det(T[i,j])\ne 0$ for $0\le i\le j\le d$. Therefore $T$ is very good by Definition 4.4. Consequently $T\in \mathcal{T}_d(\mathbb{F})$.
\end{proof}
\begin{lem}
$T$ is nice in the sense of Definition 7.5.
\end{lem}
\begin{proof}
By Lemmas 9.3, 9.8.
\end{proof}
Since $T\in \mathcal{T}_d(\mathbb{F})$, we can apply the map $b$ to $T$. Recall from Definition 6.1 that $b(T)$ is an isomorphism class of Billiard Arrays. For a Billiard Array $B\in b(T)$, we compute the $B$-value of each white 3-clique in $\bigtriangleup_{d}$.
\begin{pro}
With the above notation, the $B$-value of each white 3-clique in $\bigtriangleup_{d}$ is equal to $q^{-1}$.
\end{pro}
\begin{proof}
Apply Theorem 6.6 to the Billiard Array $B$ and evaluate (6.5) using (9.5).
\end{proof}
Combining Proposition 9.10 with Theorem 7.4 and Lemma 7.6, we obtain the following result.
\begin{cor}
$T$ is the unique nice matrix in $\mathcal{T}_d(\mathbb{F})$ such that for a Billiard Array $B\in b(T)$ the $B$-value of each white 3-clique in $\bigtriangleup_{d}$ is equal to $q^{-1}$.
\end{cor}
\begin{rmk}
Assume $q=1$. Then the expression ${n \brack k}_q$ from (9.1) becomes the usual binomial coefficient $\binom {n} {k}$. So for $0\le i\le j\le d$, $T_{ij}=\binom {j} {i}$ by (9.2) and $\det(T[i,j])=1$ by Proposition 9.6. By Proposition 9.10, for a Billiard Array $B\in b(T)$ the $B$-value of each white 3-clique in $\bigtriangleup_{d}$ is equal to $1$.
\end{rmk}

\section{Acknowledgement}
This paper was written while the author was a graduate student at the University of
Wisconsin-Madison. The author would like to thank his advisor, Paul Terwilliger, for offering
many valuable ideas and suggestions.

The author would also like to thank Kazumasa Nomura for giving this paper a close reading and offering many valuable suggestions.

\end{document}